\documentclass{amsart}
\usepackage{amsmath}
\usepackage{amssymb}
\usepackage{amsthm}
\usepackage{enumerate}
\usepackage{graphicx}
\usepackage{epstopdf}
\usepackage[colorlinks]{hyperref}
\usepackage{caption}
\usepackage{subcaption}
\usepackage{enumitem}
\usepackage{pgfplots, pgfplotstable, booktabs, colortbl, array}
\usepackage{algcompatible}
\usepackage{algorithm,algpseudocode,algorithmicx}

\DeclareCaptionType{mytype}[Algorithm]
\newenvironment{myenv}{}{}

\newtheorem{theorem}{Theorem}[section]
\newtheorem{lemma}[theorem]{Lemma}

\theoremstyle{definition}

\theoremstyle{remark}

\newtheorem{proposition}[theorem]{Proposition}

\numberwithin{equation}{section}

\DeclareMathOperator*{\argmin}{arg\,min}
\DeclareMathOperator*{\minimize}{minimize}

\begin{document}

\title{Embedding-Based Interpolation on the Special Orthogonal Group}


\author{Evan S. Gawlik}
\author{Melvin Leok}
\address{Department of Mathematics, University of California, San Diego, 9500 Gilman Drive \#0112, La Jolla, CA  92093-0112}
\curraddr{}
\email{egawlik@ucsd.edu}
\email{mleok@math.ucsd.edu}
\thanks{EG has been supported in part by the NSF under grants DMS-1411792, DMS-1345013. ML has been supported in part by the NSF under grants DMS-1010687, CMMI-1029445, DMS-1065972, CMMI-1334759, DMS-1411792, DMS-1345013.}


\date{}

\maketitle

\begin{abstract}
We study schemes for interpolating functions that take values in the special orthogonal group $SO(n)$.  Our focus is on interpolation schemes obtained by embedding $SO(n)$ in a linear space, interpolating in the linear space, and mapping the result onto $SO(n)$ via the closest point projection.  The resulting interpolants inherit both the order of accuracy and the regularity of the underlying interpolants on the linear space.  The values and derivatives of the interpolants admit efficient evaluation via either explicit formulas or iterative algorithms, which we detail for two choices of embeddings: the embedding of $SO(n)$ in the space of $n \times n$ matrices and, when $n=3$, the identification of $SO(3)$ with the set of unit quaternions.  Along the way, we point out a connection between these interpolation schemes and geodesic finite elements.  We illustrate the utility of these interpolation schemes by numerically computing minimum acceleration curves on $SO(n)$, a task which is handled naturally with $SO(n)$-valued finite elements having $C^1$-continuity.
\end{abstract}

\section{Introduction}

The special orthogonal group $SO(n)$ plays an important role in mechanics, computer graphics, and other applications, due in large part to its connection with rigid body rotations when $n=3$.  
The task of interpolating $SO(3)$-valued functions, in particular, arises in robotics~\cite{zefran1998generation}, animation~\cite{barr1992smooth,fangt1998real,ramamoorthi1997fast}, and the discretization of Cosserat continuum theories~\cite{sander2010geodesic,cao2006three}.   This paper studies schemes for interpolating such functions, with an emphasis on the case $n=3$ but with an eye toward general $n$ as well.  Our focus is on interpolation schemes obtained by embedding $SO(n)$ in a linear space, interpolating in the linear space, and mapping the result onto $SO(n)$ via the closest point projection.

The interpolants so constructed enjoy several desirable features.  First, they inherit the regularity of the underlying interpolant on the linear space.  This fact allows one to construct $SO(n)$-valued interpolants with $C^1$-continuity in a straightforward way.  Second, they inherit the order of accuracy of the underlying interpolant on the linear space.  They are also $SO(n)$-equivariant, in the sense that the interpolant transforms in the natural way when the function being interpolated is pre- or post-multiplied by an element of $SO(n)$.  Finally, their derivatives are easy to calculate, particularly when $n=3$ and quaternions are adopted to represent rotations.

Interpreted broadly, the use of an embedding for interpolation on $SO(n)$ is not without precedent.  The simplicity of this approach has attracted the attention of several prior authors~\cite{belta2002svd,gramkow2001averaging,han2008rotation,sarlette2009consensus,moakher2002means,grohs2013projection}, many of which have focused on the task of averaging rotations.
Less attention has been paid, however, to studying the derivatives of these interpolants, using these interpolants as finite elements, and studying their interpolation errors under refinement.
We give a comprehensive treatment of each of these topics in this paper.  Additionally, in our presentation of interpolation error estimates, we adopt enough generality that our results apply to a wide class of schemes for interpolating manifold-valued functions via embedding and projecting.

Alternative interpolation schemes on $SO(n)$ that do not make use of an embedding are widespread.  Perhaps the best-known example is spherical linear interpolation (abbreviated ``slerp''), in which two elements of $SO(n)$ are interpolated by the geodesic that joins them~\cite{shoemake1985animating}.  Usually this is done with the aid of quaternions when $n=3$.  This strategy leads readily to a scheme for constructing a continuous, piecewise smooth interpolant of an $SO(n)$-valued function defined on a interval.  Smoother analogues of these interpolants (called ``squads'') can be constructed using an algorithm resembling De Casteljau's algorithm, though their derivatives can be intricate to calculate~\cite{dam1998quaternions}.  A different generalization of spherical linear interpolation, which applies to manifold-valued functions defined on a domain in $\mathbb{R}^d$, $d \ge 1$, is provided by geodesic finite elements~\cite{sander2012geodesic,sander2010geodesic}. These elements, which can be designed with arbitrarily high order of accuracy~\cite{sander2015geodesic,grohs2015optimal}, are defined as solutions to a minimization problem that involves geodesic distances between the value of the interpolant and the values of the function at specified locations.  When $d=1$, they reduce to piecewise geodesics in the lowest order case.  

It is worthwhile to note that geodesic finite elements are continuous but not continuously differentiable.  In fact, a generalization of the theory of geodesic finite elements to the $C^1$ setting is not immediate, since $C^1$ finite elements typically make use of degrees of freedom that involve function values and their derivatives. For manifold-valued functions, the latter quantities belong to the manifold's tangent spaces, so a nontrivial generalization of the definition of a geodesic finite element seems necessary in order to incorporate such degrees of freedom.

Another class of interpolation strategies, which apply not only to interpolation on $SO(n)$ but also on any Lie group $G$, consists of methods that use the Lie group exponential map and its inverse to map elements of $G$ to the Lie algebra $\mathfrak{g}$ of $G$ and perform interpolation there~\cite{kim1995general,park1997smooth}.  If done carefully, interpolants having $C^1$-continuity and relatively simple derivatives can be constructed with this approach~\cite{kim1995general}.

It should be noted that the closely related but slightly simpler task of averaging rotations -- without necessarily constructing continuous or continuously differentiable interpolants of rotations -- is the subject of a vast body of literature.  A comprehensive review of this literature would be outside the scope of this paper, but a good survey is given in~\cite{hartley2013rotation}.  

The task of constructing continuously differentiable $SO(n)$-valued interpolants is much more than a pedantic exercise; it is a topic of longstanding interest in computer graphics and motion planning~\cite{kim1995general,park1997smooth,dam1998quaternions,park1995bezier}.  There, the interest is in constructing smooth motions of rigid bodies that interpolate specified orientations and, potentially, specified angular velocities.  A task of particular import is the construction of minimum acceleration curves -- smooth curves on $SO(3)$ that minimize angular acceleration in an $L^2$-sense, subject to suitable boundary conditions~\cite{barr1992smooth,ramamoorthi1997fast,gay2012invariant,park1997smooth}.  An analogous notion of optimality can be defined for curves on a Riemannian manifold.  The resulting minimizers, which can be thought of as higher-order generalizations of geodesics, are referred to as Riemannian cubics, owing to the fact that they reduce to cubic polynomials when the manifold under consideration is Euclidean~\cite{crouch1995dynamic,noakes1989cubic}.  

We show in this paper that the computation of minimum acceleration curves on $SO(n)$ is handled seamlessly with embedding-based interpolation schemes.  Since they allow one to easily construct $SO(n)$-valued finite elements with $C^1$-continuity, a conforming discretization of the minimum acceleration problem is readily obtained, leading to a finite-dimensional minimization problem.  In appropriate variables, this minimization problem is an unconstrained least squares problem, thereby admitting an efficient solution with standard algorithms such as the Levenberg-Marquardt algorithm~\cite{more1978levenberg}.  Under refinement, the numerical solution so obtained exhibits convergence to the exact solution with optimal order of accuracy.  

There are some parallels between the present work and certain subdivision schemes for manifold-valued functions, particularly those that make use of an embedding~\cite{wallner2005convergence,xie2007smoothness}.  Our results in Section~\ref{sec:theory} concerning the regularity and approximation properties of embedding-based interpolants are closely related to those established for such manifold-valued subdivision schemes.  In that context, the terms ``smoothness equivalence'' and ``approximation order equivalence'' have been used to describe the regularity and order of accuracy that these manifold-valued subdivision schemes inherit from their Euclidean counterparts~\cite{xie2011approximation,grohs2009smoothness}.

There are also parallels between the present work and geodesic finite elements.  We point out in Section~\ref{sec:geodesicFE} that if a geodesic finite element is constructed using a chordal metric -- the metric inherited from an embedding in a linear space -- then it coincides with the finite element one obtains by interpolating in the linear space with Lagrange polynomials and projecting the result onto the manifold via the closest point projection.

\subsubsection*{Organization} 
This paper is organized as follows.  In Section~\ref{sec:theory}, we define a class of interpolation operators for manifold-valued functions obtained from embedding and projecting, and we derive estimates for the error committed by these interpolants and their first derivatives.  We leave the manifold unspecified throughout Section~\ref{sec:theory}, since the arguments apply rather generally.  In Section~\ref{sec:SOn}, we specialize to the case in which the manifold under consideration is the special orthogonal group $SO(n)$.  We present interpolation schemes on $SO(n)$ based on two choices of embeddings: the embedding of $SO(n)$ in the space of $n \times n$ matrices and, when $n=3$, the identification of $SO(3)$ with the set of unit quaternions.  We derive explicit formulas and iterative algorithms for computing the values and derivatives of these interpolants.  In Section~\ref{sec:minaccel}, we illustrate the utility of these interpolation schemes by numerically computing minimum acceleration curves on $SO(n)$, a task which is handled naturally with $SO(n)$-valued finite elements having $C^1$-continuity.

\section{Embedding-Based Interpolation of Manifold-Valued Functions} \label{sec:theory}

In this section, we discuss a class of interpolation operators for manifold-valued functions obtained by embedding the manifold in a linear space, interpolating in the linear space, and mapping the result onto the manifold via the closest point projection.  We discuss several properties of these interpolants, including their pointwise accuracy, their regularity, the accuracy of their derivatives, and their connection with geodesic finite elements.

Let $M$ be a smooth Riemannian manifold embedded in $\mathbb{R}^p$, $p \ge 1$.  Let $\mathcal{D} \subset \mathbb{R}^d$, $d \ge 1$, be a compact, connected, Lipschitz domain.  Let $\mathcal{V}(\mathcal{D},\mathbb{R}^p)$ be a vector space of functions from $\mathcal{D}$ to $\mathbb{R}^p$ contained in $C(\mathcal{D},\mathbb{R}^p)$, the space of continuous functions from $\mathcal{D}$ to $\mathbb{R}^p$.  Let $\mathcal{V}_h(\mathcal{D},\mathbb{R}^p) \subset \mathcal{V}(\mathcal{D},\mathbb{R}^p)$ be a finite-dimensional subspace of $\mathcal{V}(\mathcal{D},\mathbb{R}^p)$.  Let
\[ 
\mathcal{I}_h : \mathcal{V}(\mathcal{D},\mathbb{R}^p) \rightarrow \mathcal{V}_h(\mathcal{D},\mathbb{R}^p)
\]
be a projection, hereafter referred to as an interpolation operator for $\mathbb{R}^p$-valued functions.  Our aim is to use $\mathcal{I}_h$ to construct an interpolation operator for (suitably regular) $M$-valued functions $u : \mathcal{D} \rightarrow M$ by projecting $\mathcal{I}_h u$ pointwise onto $M$.  To this end, let
\[
\mathcal{V}(\mathcal{D},M) = \{u \in \mathcal{V}(\mathcal{D},\mathbb{R}^p) : u(x) \in M \; \forall x \in \mathcal{D}\},
\]
and denote by $\|\cdot\|$ the Euclidean norm on $\mathbb{R}^p$.
In a tubular neighborhood $U \subset \mathbb{R}^p$ of $M$, the closest point projection 
\begin{align}\begin{split} \label{cpp}
\mathcal{P}_M : U  &\rightarrow M \\
u &\mapsto \argmin_{m \in M} \|m-u\|
\end{split}\end{align}
is well-defined and smooth; see~\cite[Theorem 10.19]{lee2003smooth}.  We shall abuse notation by using the same symbol $\mathcal{P}_M$ to denote the map
\begin{align*}
\mathcal{P}_M : C(\mathcal{D},U) \rightarrow C(\mathcal{D},M)
\end{align*}
which sends a continuous function $u : \mathcal{D} \rightarrow U$ to the continuous function $\mathcal{P}_M u : \mathcal{D} \rightarrow M$ given by 
\[
(\mathcal{P}_M u)(x) = \mathcal{P}_M(u(x))
\]
for every $x \in \mathcal{D}$.  
Now define
\begin{equation} \label{projectedinterp}
\mathcal{I}_{h,M} = \mathcal{P}_M \circ \left.\mathcal{I}_h\right|_{\widetilde{\mathcal{V}}(\mathcal{D},M)},
\end{equation}
where
\[
\widetilde{\mathcal{V}}(\mathcal{D},M) = \{ u \in \mathcal{V}(\mathcal{D},M) : \mathcal{I}_h u(x) \in U \; \forall x \in \mathcal{D}\}.
\]
We refer to $\mathcal{I}_{h,M}$ as an interpolation operator for $M$-valued functions, and we denote the image of $\widetilde{\mathcal{V}}(\mathcal{D},M)$ under $\mathcal{I}_{h,M}$ by $\mathcal{V}_h(\mathcal{D},M)$.
Note that the absence of a subscript $h$ on $\widetilde{\mathcal{V}}(\mathcal{D},M)$ is somewhat misleading in view of its dependence on $\mathcal{I}_h$.  We have chosen this notation to emphasize that $\widetilde{\mathcal{V}}(\mathcal{D},M)$ is, in general, an infinite-dimensional space.

\subsection{Properties of the Interpolant}

We now detail several features of the interpolation operator $\mathcal{I}_{h,M}$.  Our main observation is that many of the properties of $\mathcal{I}_{h,M}$ -- regularity and order of approximation -- are inherited from $\mathcal{I}_h$.

An immediate consequence of the definition of $\mathcal{I}_{h,M}$ and the smoothness of $\mathcal{P}_M$ is the following proposition, which leads to a simple method of constructing manifold-valued finite elements with higher regularity.
\begin{proposition}
If $u \in \widetilde{\mathcal{V}}(\mathcal{D},M)$ and $\mathcal{I}_h u \in C^k(\mathcal{D},\mathbb{R}^p)$, $k \ge 0$, then $\mathcal{I}_{h,M}u \in C^k(\mathcal{D},M)$.
\end{proposition}
\begin{proof}
Since $\mathcal{I}_h u$ is of class $C^k$ and $\mathcal{P}_M$ is of class $C^\infty$, the composition $\mathcal{I}_{h,M}u = \mathcal{P}_M \circ \mathcal{I}_h u$ is of class $C^k$~\cite[Proposition 3.2.8]{abraham2007manifolds}.
\end{proof}

To a large extent, the approximation properties of $\mathcal{I}_{h,M}$ are also inherited from $\mathcal{I}_h$. The following proposition, whose proof is notably elementary, shows that the pointwise error committed by the interpolant $\mathcal{I}_{h,M} u$ of a function $u \in \mathcal{V}(\mathcal{D},M)$ is no worse than that committed by $\mathcal{I}_h u$, up to a factor of 2.

\begin{proposition} \label{lemma:Linferror}
For any $u \in \widetilde{\mathcal{V}}(\mathcal{D},M)$ and any $x \in \mathcal{D}$,
\[
\|\mathcal{I}_{h,M} u(x) - u(x)\| \le 2 \|\mathcal{I}_h u(x) - u(x)\|.
\]
\end{proposition}
\begin{proof}
The triangle inequality and the definition of $\mathcal{I}_{h,M}$ give
\begin{align*}
\|\mathcal{I}_{h,M} u(x) - u(x)\| 
&\le \|\mathcal{P}_M \mathcal{I}_h u(x) - \mathcal{I}_h u(x)\| +  \|\mathcal{I}_h u(x) - u(x)\|.
\end{align*}
Now since $u(x) \in M$, the definition of $\mathcal{P}_M$ implies that
\[
\|\mathcal{P}_M \mathcal{I}_h u(x) - \mathcal{I}_h u(x)\| \le \|u(x) - \mathcal{I}_h u(x)\|,
\]
which proves the claim.
\end{proof}

The next proposition shows furthermore that the pointwise error in the gradient of $\mathcal{I}_{h,M} u$ is essentially of the same order as the pointwise error in the gradient of $\mathcal{I}_h u$, provided that $u$ is sufficiently smooth.  Below, we denote by $\nabla u(x) \in \mathbb{R}^{p \times d}$ and $\nabla \mathcal{P}_M(v) \in \mathbb{R}^{p \times p}$ the gradients of $u$ and $\mathcal{P}_M$ at $x \in \mathcal{D}$ and $v \in U$, respectively, viewing $u$ and $\mathcal{P}_M$ as $\mathbb{R}^p$-valued functions via the embedding $M \subset \mathbb{R}^p$.  We denote
\begin{align*}
C_0(u) &= \sup_{x \in \mathcal{D}} \|\nabla u(x)\|, \\
C_1 &= \sup_{m \in M} \|\nabla \mathcal{P}_M(m)\|, \\
C_2 &= \sup_{u_1,u_2 \in U, \atop u_1 \neq u_2} \frac{ \|\nabla \mathcal{P}_M(u_1) - \nabla \mathcal{P}_M(u_2)\| }{\|u_1-u_2\|},
\end{align*}
where, for matrices, $\|\cdot\|$ denotes any consistent matrix norm. 

\begin{proposition} \label{lemma:W1inferror}
For any $u \in \widetilde{\mathcal{V}}(\mathcal{D},M) \cap C^1(\mathcal{D},\mathbb{R}^p)$ and any $x \in \mathcal{D}$,
\[\begin{split}
\|\nabla \mathcal{I}_{h,M} u(x) - \nabla u(x) \| \le C_1 &\| \nabla \mathcal{I}_h u(x) - \nabla u(x)\| \\
&+ C_2 \|\mathcal{I}_h u(x) - u(x)\| \left( \| \nabla \mathcal{I}_h u(x) - \nabla u(x)\|   + C_0(u) \right).
\end{split}\]
\end{proposition}
\begin{proof}
The chain rule gives
\[
\nabla \mathcal{I}_{h,M} u(x) = \nabla \mathcal{P}_M(\mathcal{I}_h u(x)) \nabla \mathcal{I}_h u(x).
\]
On the other hand, since $\mathcal{P}_Mu=u$ pointwise,
\begin{align*}
\nabla u(x) 
&= \nabla (\mathcal{P}_M u)(x) \\
&= \nabla \mathcal{P}_M (u(x)) \nabla u(x).
\end{align*}
Thus,
\begin{align*}
\nabla \mathcal{I}_{h,M} u(x)  - \nabla u(x)
&= \nabla \mathcal{P}_M(\mathcal{I}_h u(x)) \nabla \mathcal{I}_h u(x) - \nabla \mathcal{P}_M (u(x)) \nabla u(x) \\
&= \left[ \nabla \mathcal{P}_M(\mathcal{I}_h u(x)) - \nabla \mathcal{P}_M (u(x)) \right] \nabla \mathcal{I}_h u(x) \\&\quad + \nabla \mathcal{P}_M (u(x)) \left[ \nabla \mathcal{I}_h u(x) - \nabla u(x) \right].
\end{align*}
It follows that
\begin{align*}
\|\nabla \mathcal{I}_{h,M} u(x)  - \nabla u(x)\|
&\le C_2 \|\mathcal{I}_h u(x) - u(x)\| \|\nabla \mathcal{I}_h u(x)\| + C_1 \|\nabla \mathcal{I}_h u(x) - \nabla u(x)\|.
\end{align*}
The conclusion then follows upon noting that
\begin{align*}
\|\nabla \mathcal{I}_h u(x)\| 
&\le \|\nabla \mathcal{I}_h u(x) - \nabla u(x)\| + \|\nabla u(x)\| \\
&\le \|\nabla \mathcal{I}_h u(x) - \nabla u(x)\| + C_0(u).
\end{align*}
\end{proof}

We remark that in typical applications, $\mathcal{I}_h$ is an interpolation operator parametrized by a discretization parameter $h$ such that, for some integer $r \ge 1$ and any sufficiently regular $u \in \mathcal{V}(\mathcal{D},\mathbb{R}^p)$, 
\begin{align*}
\sup_{x \in \mathcal{D}} \|\mathcal{I}_h u(x) - u(x)\| &= O(h^{r+1}), \\
\sup_{x \in \mathcal{D}} \|\nabla \mathcal{I}_h u(x) - \nabla u(x)\| &= O(h^r)
\end{align*}
as $h \rightarrow 0$.  In such a setting, the preceding two propositions imply that $\mathcal{I}_{h,M}u$ enjoys the same order of accuracy for sufficiently regular $u\in \widetilde{\mathcal{V}}(\mathcal{D},M)$.  That is,
\begin{align*}
\sup_{x \in \mathcal{D}} \|\mathcal{I}_{h,M} u(x) - u(x)\| &= O(h^{r+1}), \\
\sup_{x \in \mathcal{D}} \|\nabla \mathcal{I}_{h,M} u(x) - \nabla u(x)\| &= O(h^r).
\end{align*}

\subsection{Relationship to Geodesic Finite Elements} \label{sec:geodesicFE}
We conclude this section by pointing out a relationship between the interpolation operator $\mathcal{I}_{h,M}$ and geodesic finite elements~\cite{sander2012geodesic,sander2015geodesic,sander2010geodesic}.  Given a partition of a polyhedral domain $\mathcal{D}$ into simplices, an $r^{th}$-order geodesic finite element is an interpolant of an $M$-valued function $u : \mathcal{D} \rightarrow M$ whose value at any $x \in \mathcal{D}$ is given by
\begin{equation} \label{geodesicFE}
\argmin_{m \in M} \sum_i \varphi_i(x) \mathrm{dist}(m,u_i)^2,
\end{equation}
where $\{\varphi_i\}_i$ is a basis of Lagrange polynomials~\cite[Section 1.2.3]{ern2004theory} of degree $\le r$ on a simplex $K \subseteq \mathcal{D}$ containing $x$, and $\{u_i\}_i$ are the values of $u$ at the corresponding nodes.  Here, $\mathrm{dist} : M \times M \rightarrow \mathbb{R}$ denotes a distance function on $M$, which is typically defined intrinsically, without appealing to an embedding.  In the event that an embedding $M \subset \mathbb{R}^p$ is used to define a distance function via $\mathrm{dist}(m_1,m_2) = \|m_1-m_2\|$, the resulting geodesic finite element reduces to an interpolant of the form~(\ref{projectedinterp}).  This fact is detailed in the following proposition.

\begin{proposition} \label{lemma:geodesicFE}
Let $u \in \mathcal{V}(\mathcal{D},M)$.  Let $\{\varphi_i\}_i$ be a basis of Lagrange polynomials of degree $\le r$ on a simplex $K \subseteq \mathcal{D}$, and  let $\{u_i\}_i$ be the values $u$ at the corresponding nodes.  Then for any $x \in K$,
\[
\argmin_{m \in M} \sum_i \varphi_i(x) \|m-u_i\|^2 = \mathcal{P}_M \left( \sum_i \varphi_i(x) u_i \right),
\]
provided that $\sum_i \varphi_i(x) u_i$ belongs to the tubular neighborhood $U \supset M$ on which $\mathcal{P}_M$ is defined.
\end{proposition}
\begin{proof}
Since Lagrange polynomials interpolate constant scalar-valued functions exactly, $\sum_i \varphi_i(x) = 1$.  Thus,
\begin{align*}
\sum_i \varphi_i(x) \|m-u_i\|^2 
&= \sum_i \varphi_i(x) \left( \langle m, m \rangle - 2 \langle m, u_i \rangle + \langle u_i, u_i \rangle \right) \\
&= \langle m, m \rangle - 2 \langle m, \sum_i \varphi_i(x) u_i \rangle + \sum_i \varphi_i(x) \langle u_i, u_i \rangle \\
&= \left\| m - \sum_i \varphi_i(x) u_i\right\|^2 - \left\| \sum_i \varphi_i(x) u_i \right\|^2  + \sum_i \varphi_i(x) \langle u_i, u_i \rangle,
\end{align*}
where $\langle \cdot, \cdot \rangle$ denotes the Euclidean inner product.  Since the latter two terms in the last line above are independent of $m$, it follows that any minimizer of $\sum_i \varphi_i(x) \|m-u_i\|^2$ must be a minimizer of $\left\| m - \sum_i \varphi_i(x) u_i\right\|$, and vice versa.
\end{proof}

\section{Interpolation on the Special Orthogonal Group} \label{sec:SOn}

In this section, we specialize the interpolation operators considered in Section~\ref{sec:theory} to the case in which
\[
M=SO(n) = \{Q \in \mathbb{R}^{n \times n} : Q^T Q=I, \det Q > 0\},
\]
the special orthogonal group.  We consider two choices of embeddings: the embedding of $SO(n)$ in $\mathbb{R}^{n \times n}$, and, when $n=3$, the identification of $SO(3)$ with the set of unit quaternions.

\subsection{Embedding in $\mathbb{R}^{n \times n}$} \label{sec:Rnxn}

Consider the embedding of $SO(n)$ in $\mathbb{R}^{n \times n}$ (equipped with the Frobenius norm $\|A\|_F^2 = \mathrm{Tr}(A^T A)$).
The closest point projection~(\ref{cpp}) of a matrix $A \in \mathbb{R}^{n \times n}$ onto $SO(n)$ is given by
\[
\mathcal{P}_{SO(n)}(A) = \argmin_{Q \in SO(n)} \|Q-A\|_F.
\]
It is well-known~\cite[Theorem 1]{Fan1955} that if $\det A \neq 0$, then $\|Q-A\|_F$ has a unique minimizer among all $Q \in O(n) = \{Q \in \mathbb{R}^{n \times n} : Q^T Q = I\}$.  This minimizer is given by the orthogonal factor in the polar decomposition 
\[
A = QY, \quad Q \in O(n), \, Y \in \mathrm{Sym}_+(n),
\]
where $\mathrm{Sym}_+(n)$ denotes the space of symmetric positive definite $n \times n$ matrices.  If $\det A > 0$, then the minimizer in fact belongs to $SO(n)$.  From this it follows that $\mathcal{P}_{SO(n)}$ is well-defined and smooth on $GL_+(n) = \{A \in \mathbb{R}^{n \times n} : \det A > 0\}$, the connected component  of the identity in the general linear group.
The invariance of the Frobenius norm under pre- and post-multiplication by orthogonal matrices implies that 
\begin{equation} \label{biinvariantpolardecomp}
\mathcal{P}_{SO(n)} (UAV) = U \left( \mathcal{P}_{SO(n)} A \right) V, \quad \forall U,V \in SO(n), \, \forall A \in GL_+(n).
\end{equation}
 
\subsubsection*{A Continuous Interpolant} 
We may construct a continuous $SO(n)$-valued interpolant $\mathcal{I}_{h,SO(n)} R$ of a function $R \in C(\mathcal{D},SO(n))$ as follows.  Assume that $\mathcal{D} \subset \mathbb{R}^d$ is a polyhedral domain that has been partitioned into $d$-dimensional simplices with maximum diameter $h$, and assume that these simplices have the property that the intersection of any two of them is either empty or a common $(d-1)$-dimensional face.
On each simplex $K \subseteq \mathcal{D}$, let $\{\varphi_i\}_i$ be a basis of Lagrange polynomials of degree $\le r$, and  let $\{R_i\}_i \subset SO(n)$ be the values of $R$ at the corresponding nodes.  Then for any $x \in K$, we may set
\begin{equation} \label{interp0}
\mathcal{I}_{h,SO(n)} R(x) = \mathcal{P}_{SO(n)} \left( \sum_i \varphi_i(x) R_i \right),
\end{equation}
provided that the determinant of $\sum_i \varphi_i(x) R_i$ is positive.  By Proposition~\ref{lemma:geodesicFE}, this interpolant is equivalent to a geodesic finite element with respect to the chordal metric on $SO(n)$, which defines the distance between two matrices $P,Q \in SO(n)$ as the Frobenius norm $\|P-Q\|_F$ of their difference.  In other words, an equivalent definition of $\mathcal{I}_{h,SO(n)} R$ is
\begin{equation} \label{chordalSOn}
\mathcal{I}_{h,SO(n)} R(x) = \argmin_{Q \in SO(n)} \sum_i \varphi_i(x) \|Q-R_i\|_F^2.
\end{equation}
The equivalence of~(\ref{interp0}) and~(\ref{chordalSOn}) is a fact that has been observed previously in other contexts; see, for instance,~\cite{moakher2002means}.

Propositions~\ref{lemma:Linferror}-\ref{lemma:W1inferror} ensure that this interpolation operator inherits the optimal approximation properties enjoyed by componentwise polynomial interpolation of $\mathbb{R}^{n \times n}$-valued functions.  Namely, upon viewing the simplicial partition as a member of a family of shape-regular partitions parametrized by $h$, we have
\begin{align}
\sup_{x \in \mathcal{D}} \|\mathcal{I}_{h,SO(n)} R(x) - R(x)\| &= O(h^{r+1}), \label{interp0approx0} \\
\sup_{x \in \mathcal{D}} \|\nabla \mathcal{I}_{h,SO(n)} R(x) - \nabla R(x)\| &= O(h^r) \label{interp0approx1}
\end{align}
for any sufficiently regular $R : \mathcal{D} \rightarrow SO(n)$ whose componentwise interpolant has positive determinant everywhere.  

By~(\ref{biinvariantpolardecomp}), this interpolation operator is $SO(n)$-equivariant, in the sense that if $U,V \in SO(n)$ and $\widetilde{R}(x) = U R(x) V$ for every $x \in \mathcal{D}$, then $\mathcal{I}_{h,SO(n)} \widetilde{R}(x) = U \left(\mathcal{I}_{h,SO(n)} R(x)\right) V$ for every $x \in \mathcal{D}$.

\subsubsection*{A Continuously Differentiable Interpolant} 
A continuously differentiable $SO(n)$-valued interpolant can be constructed by using $C^1$ shape functions rather than Lagrange polynomials.  For concreteness, consider the case in which $\mathcal{D}=[0,T]$ is an interval and $R \in C^1([0,T],SO(n))$.  Let $0 = t_0 < t_1 < \dots < t_N = T$ be a partition of $[0,T]$ into subintervals having maximum length $h$.  Since $\mathcal{D}$ is one-dimensional, we will use the letter $t$ as the independent variable here rather than $x$, and denote derivatives with respect to $t$ with overdot notation. Let $\phi_i : [0,1] \rightarrow \mathbb{R}$ and $\psi_i : [0,1] \rightarrow \mathbb{R}$, $i=0,1$, be cubic polynomials satisfying
\begin{align*}
\frac{d^l}{dt^l} \phi_i(j) &= \delta_{ij}\delta_{0l}, \\
\frac{d^l}{dt^l} \psi_i(j) &= \delta_{ij}\delta_{1l}.
\end{align*}
for each $i,j,l \in \{0,1\}$, where $\delta_{ij}$ denotes the Kronecker delta.  Explicitly,
\begin{alignat}{3}
\phi_0(t) &= 2t^3-3t^2+1, &\quad \phi_1(t) &= -2t^3+3t^2, \label{hermite1} \\
\psi_0(t) &= t^3-2t^2+t, &\quad \psi_1(t) &= t^3-t^2. \label{hermite2}
\end{alignat}
These are the Hermite cubic basis functions, so that
\[\begin{split}
\mathcal{I}_h R(t) = \sum_{i=0}^1 \phi_i\left(\frac{t-t_k}{t_{k+1}-t_k}\right) R(t_{k+i}) + (t_{k+1}-t_k) \psi_i\left(\frac{t-t_k}{t_{k+1}-t_k}\right) \dot{R}(t_{k+i}), \\ t \in [t_k,t_{k+1}], 0 \le k < N &
\end{split}\]
defines an interpolant of $R$ belonging to $C^1([0,T],\mathbb{R}^{n \times n})$.
Projecting this interpolant onto $SO(n)$ defines a interpolant of $R$ belonging to $C^1([0,T],SO(n))$ given by
\begin{equation} \label{interp1}
\begin{split}
\mathcal{I}_{h,SO(n)} R(t) &= \mathcal{P}_{SO(n)} \left( \sum_{i=0}^1 \phi_i\left(\frac{t-t_k}{t_{k+1}-t_k}\right) R(t_{k+i}) \right. \\&\left.\quad\;+ (t_{k+1}-t_k) \psi_i\left(\frac{t-t_k}{t_{k+1}-t_k}\right) \dot{R}(t_{k+i}) \right), \quad t \in [t_k,t_{k+1}], \, 0 \le k < N.
\end{split}
\end{equation}
Unlike~(\ref{interp0}), this interpolant is not an instance of a geodesic finite element~(\ref{geodesicFE}).  However, it is $SO(n)$-equivariant in view of~(\ref{biinvariantpolardecomp}).
Propositions~\ref{lemma:Linferror}-\ref{lemma:W1inferror}, together with well-known properties of Hermite cubic interpolation, ensure that the interpolant~(\ref{interp1}) enjoys approximation error estimates of the form~(\ref{interp0approx0}-\ref{interp0approx1}) with $r=3$.

\subsubsection{First-Order Derivatives of the Interpolant} \label{sec:deriv}

In this section, we study the derivatives of the interpolants~(\ref{interp0}) and~(\ref{interp1}).  Without loss of generality, we focus on differentiating the orthogonal factor $Q(t)$ in the polar decomposition
\begin{equation} \label{polardecomp}
A(t)=Q(t)Y(t), \quad Q(t) \in O(n), \, Y(t) \in \mathrm{Sym}_+(n),
\end{equation}
of an $n \times n$ nonsingular matrix $A(t)$ whose entries depend smoothly on a single parameter $t$.  Derivatives of the interpolant~(\ref{interp1}) will follow readily.  On the other hand, derivatives of the interpolant~(\ref{interp0}) in each of the coordinate directions $e_j$, $j=1,2,\dots,d$, can be obtained by considering the matrix
\[
A(t) = \sum_i \varphi_i(x+te_j) R_i,
\]
and noting that if $Q(t)$ is related to $A(t)$ according to~(\ref{polardecomp}), then
\[
\frac{\partial}{\partial x_j}\mathcal{I}_{h,SO(n)} R(x) = \left.\frac{d}{dt}\right|_{t=0} Q(t).
\]

To compute the derivative of $Q(t)$, differentiate the decomposition~(\ref{polardecomp}) to obtain
\begin{equation} \label{Adot}
\dot{A} = \dot{Q}Y + Q\dot{Y}.
\end{equation}
Next, multiply by $Q^T = Q^{-1}$ to obtain
\begin{equation} \label{QTAdot}
Q^T \dot{A} = \Omega Y + \dot{Y},
\end{equation}
where $\Omega = Q^T \dot{Q} \in \mathfrak{so}(n) = \{ \Omega \in \mathbb{R}^{n \times n} : \Omega+\Omega^T = 0\}$.  Since $Y$ and $\dot{Y}$ are symmetric and $\Omega$ is skew-symmetric, the skew-symmetric part of~(\ref{QTAdot}) reads
\begin{equation} \label{syl1}
 Y\Omega + \Omega Y  = Q^T \dot{A} - \dot{A}^T Q.
\end{equation}
Given $Q$, $Y$, and $\dot{A}$, this is a Lyapunov equation for the unknown $\Omega$ which, by the positive-definiteness of $Y$, has a unique solution. In principle, this Lyapunov equation can be solved numerically using standard algorithms~\cite{bartels1972solution,golub1979hessenberg}.  If $n$ is large, however, more efficient methods are available, as we explain toward the end of this section.

An alternative means of finding $\dot{Q}$ is based on differentiating the relation
\[
Y(t)^2 = A(t)^T A(t).
\]
One finds that the symmetric matrix $\dot{Y}$ must satisfy a Lyapunov equation
\begin{equation} \label{syl2}
Y \dot{Y} + \dot{Y} Y = \dot{A}^T A + A^T \dot{A}.
\end{equation}
Upon solving for $\dot{Y}$, the value of $\dot{Q}$ is immediate from~(\ref{Adot}):
\begin{equation} \label{Qdot}
\dot{Q} = (\dot{A}-Q\dot{Y}) Y^{-1}.
\end{equation}

The remainder of this section is devoted to two independent tasks.  First, we derive more explicit formulas for $\dot{Q}$ in special cases.  Second, we develop efficient methods for computing $\dot{Q}$ when explicit formulas are unavailable.

\subsubsection*{Explicit Formula at the Nodes}
If $A(t_0)$ is orthogonal for some $t_0$, then the derivative of the orthogonal factor $Q(t)$ in its polar decomposition~(\ref{polardecomp}) simplifies considerably at $t=t_0$.  In this setting, $Q(t_0)=A(t_0)$ and $Y(t_0)=I$, allowing~(\ref{syl1}) and~(\ref{syl2}) to be rewritten in the form
\begin{align*}
\dot{Q} &= Q \, \mathrm{skew}(A^{-1}\dot{A}), \\
\dot{Y} &= Y \, \mathrm{sym}(A^{-1}\dot{A})
\end{align*}
at $t=t_0$, where $\mathrm{skew}(B) = \frac{1}{2}(B-B^T)$ and $\mathrm{sym}(B) = \frac{1}{2}(B+B^T)$ denote the 
skew-symmetric and symmetric parts, respectively, of a square matrix $B$.

A consequence of this observation is that the derivatives of the interpolant~(\ref{interp0}) at the nodes $\{v_i\}_i$ of a Lagrangian finite element on a simplex $K$ are explicitly computable.  Namely, since $R(v_k)=R_k \in SO(n)$,
\[
\frac{\partial}{\partial x_j}\mathcal{I}_{h,SO(n)}R(v_k) =  R_k \,\mathrm{skew}\left( R_k^T \left(\sum_{i=1}^m \frac{\partial \varphi_i}{\partial x_j}(v_k) R_i \right) \right)
\]
for each $k$ and each $j=1,2,\dots,d$.

Likewise, for the interpolant~(\ref{interp1}),
\begin{align*}
\frac{d}{d t}\mathcal{I}_{h,SO(n)}R(t_k) 
&=  R(t_k) \,\mathrm{skew}\left( R(t_k)^T \dot{R}(t_k) \right) \\
&= \dot{R}(t_k).
\end{align*}

\subsubsection*{Explicit Formula in Three Dimensions}
In dimension $n=3$, explicit formulas for the derivatives of the orthogonal factor in the polar decomposition of a smooth matrix-valued function $A : \mathbb{R} \rightarrow \mathbb{R}^{3 \times 3}$ are known.  Namely, if~(\ref{polardecomp}) is the polar decomposition of $A(t) \in \mathbb{R}^{3 \times 3}$, then~\cite[p. 181]{Chen1993}
\[
\dot{Q}(t) = 2Q(t) \left(\det Z(t)\right)^{-1} Z(t) \, \mathrm{skew}\left( A(t)^{-1} \dot{A}(t)  Y(t) \right) Z(t) ,
\]
where
\[
Z(t) = \mathrm{Tr}(Y(t)) I - Y(t).
\]

\subsubsection*{Explicit Formula for Linear Univariate Polynomials}
If $A(t)$ is the componentwise linear interpolant of two matrices in $SO(n)$, $n \ge 1$, then explicit formulas for the derivatives of $Q(t)$ and $Y(t)$ are also obtainable. To illustrate this fact, consider the interpolant~(\ref{interp0}) on a one-dimensional domain $\mathcal{D}=[0,T]$ using piecewise linear polynomials ($r=1$).  
In this setting, the interpolant~(\ref{interp0}) reduces to
\begin{equation} \label{linearinterp}
\mathcal{I}_{h,SO(n)} R(t) = \mathcal{P}_{SO(n)} \left( \frac{t_{k+1}-t}{t_{k+1}-t_k} R_k + \frac{t-t_k}{t_{k+1}-t_k} R_{k+1} \right), \quad t \in [t_k,t_{k+1}], \, 0 \le k < N,
\end{equation}
where $0 = t_0 < t_1 < \dots t_N = T$ is a partition of $[0,T]$, $R_k = R(t_k)$, and we have used the letter $t$ instead of $x$ to denote the independent variable.  The following lemma gives a formula for the derivative of $\mathcal{I}_{h,SO(n)} R(t)$ on each interval $[t_k,t_{k+1}]$.  In an abuse of notation, we set $t_k=0$, $t_{k+1}=h$, and $k=0$ in what follows.
\begin{lemma}
Let $R_0,R_1 \in SO(n)$ and $h>0$ be given. For each $t \in [0,h]$, let
\begin{equation} \label{Alinear}
A(t) = \frac{h-t}{h} R_0 + \frac{t}{h} R_1.
\end{equation}
Let $A(t)=Q(t)Y(t)$ be the polar decomposition of $A(t)$, where $Q(t)$ is orthogonal and $Y(t)$ is symmetric positive definite.  Then
\begin{align}
\dot{Y}(t) &= Y(t) \, \mathrm{sym}(A(t)^{-1} \dot{A}(t)), \label{dY} \\
\dot{Q}(t) &= Q(t) \, \mathrm{skew}(A(t)^{-1} \dot{A}(t)) \label{dR}.
\end{align}
\end{lemma}
\begin{proof}
We may assume without loss of generality that $R_0 = I$, so that
\begin{equation} \label{Awlog}
A(t) = \frac{h-t}{h} I + \frac{t}{h} R_1.
\end{equation}
Then $A(t)$ commutes with $A(t)^T$, from which it follows~\cite{higham2008functions} that $Q(t)$ commutes with $Y(t)$.

We claim that if furthermore $\dot{Y}(t)$ can be shown to commute with $Q(t)$ and $Y(t)$, then equations~(\ref{dY}-\ref{dR}) follow readily.  Indeed, if this is the case, then differentiating the relation $A(t)=Y(t)Q(t)$ and pre-multiplying by $A(t)^{-1}=Q(t)^T Y(t)^{-1}$ gives
\begin{align}
A^{-1} \dot{A}
&= Q^T Y^{-1} (Y \dot{Q} + \dot{Y} Q) \nonumber \\
&= Q^T \dot{Q} + Q^T Y^{-1} \dot{Y} Q \nonumber \\
&= Q^T \dot{Q} + Q^T Q Y^{-1} \dot{Y} \nonumber \\
&= Q^T \dot{Q} + Y^{-1} \dot{Y}. \label{AinvAprime}
\end{align}
Since $Q(t)$ is orthogonal for all $t$, $Q^T \dot{Q}$ is skew-symmetric.  On the other hand, since $Y^{-1}$ and $\dot{Y}$ are symmetric and commute with one another, $Y^{-1} \dot{Y}$ is symmetric.  These observations lead to~(\ref{dY}-\ref{dR}).

It remains to check that $\dot{Y}$ commutes with $Q$ and $Y$.  To do so, note first that by~(\ref{Awlog}),
\begin{align}
Y(t)^2
&= A(t)^T A(t)\nonumber \\
&= \left[ \left(\frac{h-t}{h}\right)^2 + \left(\frac{t}{h}\right)^2 \right] I + \left(\frac{h-t}{h}\right) \left(\frac{t}{h}\right) (R_1+R_1^T). \label{Ysoln}
\end{align}
In particular, $Y$ is of the form $Y = (\alpha I + \beta (R_1+R_1^T))^{1/2}$ with $\alpha$ and $\beta$ scalars, so $Y$ commutes with $R_1+R_1^T$.
Furthermore, $\dot{Y}$ is a symmetric matrix satisfying
\begin{align*}
Y(t)\dot{Y}(t) + \dot{Y}(t) Y(t) 
&= \dot{A}(t)^T A(t) + A(t)^T \dot{A}(t). \\
&= \left(\frac{h-2t}{h^2} \right) (R_1 + R_1^T - 2I).
\end{align*}
A direct calculation, invoking the commutativity of $Y$ and $R_1+R_1^T$, confirms that the solution to this equation is
\begin{equation} \label{Yprimesoln}
\dot{Y}(t) = \left(\frac{h-2t}{2h^2} \right) Y(t)^{-1} (R_1 + R_1^T - 2I).
\end{equation}
Since $\dot{Y}$ is of the form $\dot{Y} = Y^{-1}(\alpha I + \beta Y^2)$ with $\alpha$ and $\beta$ scalars, and since $Y$ commutes with $Q$, it follows that $\dot{Y}$ commutes with both $Q$ and $Y$.
\end{proof}
Written more explicitly, the preceding lemma shows that if $A(t)$ is of the form~(\ref{Alinear}) and $Q(t)=\mathcal{P}_{SO(n)}(A(t))$, then
\begin{equation} \label{djRlinear}
\dot{Q}(t) =  Q(t) \, \mathrm{skew}\left(  \left(\frac{h-t}{h} R_0 + \frac{t}{h} R_1 \right)^{-1} \left( \frac{R_1-R_0}{h} \right) \right).
\end{equation}
In particular,
\begin{equation} \label{RTRdotskew}
\dot{Q}(0) =  R_0 \, \mathrm{skew}\left( R_0^T \left( \frac{R_1-R_0}{h} \right) \right)
\end{equation}
and
\[
\dot{Q}(h) =  R_1 \, \mathrm{skew}\left( R_1^T \left( \frac{R_1-R_0}{h} \right) \right).
\]
In addition,
\begin{align}
Q(h/2)^T \dot{Q}(h/2) 
&=  \mathrm{skew}\left( \left(\frac{R_0+R_1}{2}\right)^{-1} \left( \frac{R_1-R_0}{h} \right) \right) \nonumber \\
&= \frac{2}{h} \mathrm{skew}\left( \left(I+R_0^T R_1\right)^{-1} \left( R_0^T R_1 - I \right) \right) \nonumber \\
&= \frac{1}{h} \mathrm{cay}^{-1} (R_0^T R_1), \label{RTRdotcay}
\end{align}
where 
\begin{align*}
\mathrm{cay} : \mathfrak{so}(n) &\rightarrow SO(n) \\ 
\Omega &\mapsto \left(I-\frac{\Omega}{2}\right)^{-1} \left( I+\frac{\Omega}{2} \right)
\end{align*}
denotes the Cayley transform, and
\[
\mathrm{cay}^{-1}(R) = 2(I+R)^{-1}(R-I)
\]
denotes its inverse.

\subsubsection*{Iterative Computation of the First-Order Derivatives}
We now consider cases in which explicit formulas for the derivative of the orthogonal factor $Q(t)$ in the polar decomposition~(\ref{polardecomp}) are unavailable.  If this is the case, several numerical algorithms can be used to compute $\dot{Q}$.  

Suppose, for instance, that the polar decomposition~(\ref{polardecomp}) has been computed.  Then one can obtain $\dot{Q}$ by solving the Lyapunov equation~(\ref{syl1}) for $\Omega$ and computing $\dot{Q}=Q\Omega$.  Alternatively, one can solve the Lyapunov equation~(\ref{syl2}) for $\dot{Y}$ and compute $\dot{Q}$ via~(\ref{Qdot}).  Algorithms for the solution of Lyapunov equations, however, are generally expensive for large $n$, having computational cost (measured in floating point operations) close to an order of magnitude more than the cost of inverting a matrix~\cite{bartels1972solution,golub1979hessenberg}.

A more computationally efficient approach for large $n$ leverages iterative algorithms for computing the polar decomposition.  Such algorithms typically adopt fixed-point iterations of the form
\begin{equation} \label{iteration}
X_{k+1} = g(X_k), \quad X_0 = A,
\end{equation}
whose iterates $X_k$ tend to $Q$ as $k \rightarrow \infty$.
Two examples are the Newton iteration, which uses
\begin{equation} \label{Newton}
g(X) = \frac{1}{2}(X + X^{-T}),
\end{equation}
and the Newton-Schulz iteration, which uses
\begin{equation} \label{NewtonSchulz}
g(X) = \frac{1}{2} X (3I - X^T X).
\end{equation}
It is known that the Newton iteration converges quadratically to the orthogonal factor $Q$ in the polar decomposition~(\ref{polardecomp}) of $A$ for any nonsingular $n \times n$ matrix $A$~\cite[Theorem 8.12]{higham2008functions}, while the Newton-Schulz iteration converges quadratically to $Q$ provided that every singular value of $A$ lies in the interval $(0,\sqrt{3})$~\cite[Problem 8.20]{higham2008functions}.

Formally, we can differentiate the iteration~(\ref{iteration}) with respect to $t$ and obtain an algorithm for computing both $Q$ and $\dot{Q}$.  Denoting $E_k = \dot{X}_k$, the general form of such an algorithm reads
\begin{alignat}{3}
X_{k+1} &= g(X_k), &\quad& X_0 = A, \label{Xupdate} \\
E_{k+1} &= L_g(X_k,E_k), &\quad& E_0 = \dot{A}, \label{Eupdate}
\end{alignat}
where $L_g(X,E)$ denotes the Fr\'{e}chet derivative of $g$ at $X$ in the direction $E$.  It is shown in~\cite{Gawlik2016b} that under rather general circumstances, the coupled iteration~(\ref{Xupdate}-\ref{Eupdate}) produces iterates $X_k$ and $E_k$ that converge to $Q$ and $\dot{Q}$, respectively.  For the case in which $g(X)$ is given by~(\ref{Newton}), the resulting algorithm reads
\begin{alignat}{3}
X_{k+1} &= \frac{1}{2}(X_k + X_k^{-T}), &\quad& X_0 = A, \label{XupdateNewton} \\
E_{k+1} &= \frac{1}{2}(E_k - X_k^{-T} E_k^T X_k^{-T}), &\quad& E_0 = \dot{A}, \label{EupdateNewton}
\end{alignat}
When $g(X)$ is given by~(\ref{NewtonSchulz}), the algorithm reads
\begin{alignat}{3}
X_{k+1} &= \frac{1}{2}X_k(3I-X_k^T X_k), &\quad& X_0 = A, \label{XupdateNewtonSchulz} \\
E_{k+1} &= \frac{1}{2}E_k(3I-X_k^T X_k) - \frac{1}{2}X_k(E_k^T X_k + X_k^T E_k), &\quad& E_0 = \dot{A}. \label{EupdateNewtonSchulz}
\end{alignat}

\subsubsection{Higher-Order Derivatives of the Interpolant}

In some applications, such as those addressed in Section~\ref{sec:minaccel}, higher-order derivatives of the interpolants~(\ref{interp0}) and~(\ref{interp1}) are desired.  Here, we focus on computing $\ddot{Q} = \frac{d}{dt}\dot{Q}$, where $Q(t)$ is the orthogonal factor in the polar decomposition~(\ref{polardecomp}) of a matrix $A(t)$.  We also show how to compute the derivatives of $Q$, $\dot{Q}$, and $\ddot{Q}$ with respect to a parameter, assuming that $A(t)$ depends smoothly on an additional parameter which we will call $\varepsilon$.

To compute $\ddot{Q}$, consider the iteration~(\ref{XupdateNewton}-\ref{EupdateNewton}) for computing $Q$ and $\dot{Q}$.  Differentiating~(\ref{EupdateNewton}) with respect to $t$ and setting $F_k = \dot{E}_k = \ddot{X}_k$ leads formally to the following iteration for computing $\ddot{Q}$:
\begin{equation} \label{Fupdate}
F_{k+1} = \frac{1}{2} (F_k - X_k^{-T} F_k^T X_k^{-T} + 2X_k^{-T} E_k^T X_k^{-T} E_k^T X_k^{-T}), \quad F_0 = \ddot{A}.
\end{equation}
The combined iteration (that is, (\ref{XupdateNewton}), (\ref{EupdateNewton}), and~(\ref{Fupdate})),  in terms of $W_k := X_k^{-1} E_k$ and $T_k := X_k^{-1} F_k$, reads
\begin{alignat}{3}
X_{k+1} &= \frac{1}{2}(X_k + X_k^{-T}), &\quad& X_0 = A, \label{Xupdate2} \\
W_{k+1} &= \frac{1}{2}X_{k+1}^{-1} (X_k W_k - X_k^{-T} W_k^T), &\quad& W_0 = A^{-1}\dot{A}, \label{Wupdate2} \\
T_{k+1} &= \frac{1}{2}X_{k+1}^{-1} \left(X_k T_k - X_k^{-T} (T_k - 2W_k^2)^T\right), &\quad& T_0 = A^{-1}\ddot{A}. \label{Tupdate2}
\end{alignat}

If furthermore $A(t)$ depends smoothly on a parameter $\varepsilon$, then a similar argument can be used to construct iterative schemes for computing the derivatives of $Q$, $\dot{Q}$, and $\ddot{Q}$ with respect to $\varepsilon$.  Denote $\delta=\frac{\partial}{\partial\varepsilon}$, $U_k = X_k^{-1} \delta X_k$, $V_k = X_k^{-1} \delta \dot{X}_k$, and $Z_k = X_k^{-1} \delta \ddot{X}_k$.  A straightforward, but tedious, calculation shows that $U_k$, $V_k$, and $Z_k$ satisfy the recursions
\begin{alignat}{3}
U_{k+1} &= \frac{1}{2}X_{k+1}^{-1} (X_k U_k - X_k^{-T} U_k^T), &\quad& U_0 = A^{-1}\delta A, \label{Uupdate2} \\
V_{k+1} &= \frac{1}{2}X_{k+1}^{-1} \left(X_k V_k - X_k^{-T} (V_k - W_k U_k - U_k W_k)^T \right), &\quad& V_0 = A^{-1}\delta \dot{A}, \label{Vupdate2} \\
Z_{k+1} &= \frac{1}{2}X_{k+1}^{-1} \Big(X_k Z_k - X_k^{-T} \big[Z_k + 2W_k (W_k U_k - U_k W_k - V_k) &&\nonumber\\&\hspace{0.8in}+ 2(U_k W_k - V_k) W_k - U_k T_k -T_k U_k )\big]^T \Big), &\quad& Z_0 = A^{-1}\delta  \ddot{A}, \label{Zupdate2}
\end{alignat}
Our numerical experiments suggest that as $k \rightarrow \infty$, the matrices $U_k$, $V_k$, and $Z_k$ tend to $Q^T \delta Q$, $Q^T \delta \dot{Q}$, and $Q^T \delta \ddot{Q}$, respectively, although a justification of this observation would require showing that $\lim_{k\rightarrow \infty}$ commutes with differentiation.

\subsubsection{Remarks}

We conclude our discussion of the interpolants~(\ref{interp0}) and~(\ref{interp1}) with a couple of remarks.

\subsubsection*{Relationship with Variational Integrators}
Equations~(\ref{RTRdotskew}) and~(\ref{RTRdotcay}) demonstrate a relationship between the interpolant~(\ref{linearinterp}) and certain \emph{variational integrators} for rigid body dynamics~\cite{marsden2001discrete,lee2005lie}.  A variational integrator for rigid body dynamics is a numerical integrator obtained by discretizing Hamilton's principle, which states that the evolution of a rigid body's configuration $R(t) \in SO(3)$ extremizes
\begin{equation} \label{action}
\int_0^T \ell(R(t)^T \dot{R}(t)) \, dt
\end{equation}
among all curves $R : [0,T] \rightarrow SO(3)$ with fixed endpoints.  Here, $\ell : \mathfrak{so}(3) \rightarrow \mathbb{R}$ denotes the (reduced) Lagrangian: the body's kinetic energy minus its potential energy.  Two common discretizations of~(\ref{action}) are~\cite{bou2009hamilton}
\[
\int_0^T \ell(R(t)^T \dot{R}(t)) \, dt \approx \sum_{k=0}^{N-1} h \, \ell \left( \mathrm{skew}\left( R_k^T \left( \frac{R_{k+1}-R_k}{h} \right) \right) \right)
\]
and
\[
\int_0^T \ell(R(t)^T \dot{R}(t)) \, dt \approx \sum_{k=0}^{N-1} h \, \ell \left( \frac{1}{h}\mathrm{cay}^{-1}(R_k^T R_{k+1}) \right).
\]
In view of~(\ref{RTRdotskew}) and~(\ref{RTRdotcay}), these are nothing more than rectangle-rule and midpoint-rule approximations, respectively, to
\[
\int_0^h \ell(Q(t)^T \dot{Q}(t)) \, dt,
\]
where $Q(t) = \mathcal{I}_{h,SO(3)}R(t)$ denotes the $1^{st}$-order interpolant~(\ref{linearinterp}) of $\{R_k\}_{k=0}^N$ on  a uniform grid $t_k = kh$, $k=0,1,\dots,N$.

\subsubsection*{Superconvergence to Geodesics} 
Interestingly, the $1^{st}$-order interpolant~(\ref{linearinterp}) provides a superconvergent approximation of geodesics with respect to the canonical bi-invariant metric on $SO(n)$.  This fact is detailed in the following lemma, whose proof can be found in~\cite{Gawlik2016}.

\begin{lemma}
Let $R_0 \in SO(n)$, let $K \in \mathbb{R}^{n \times n}$ be an antisymmetric matrix, and let $R_1 = R_0 e^{hK}$.  For each $t \in [0,h]$, let 
\[
\left(\frac{h-t}{h}\right) R_0 + \left(\frac{t}{h}\right) R_1 = Q(t) Y(t)
\]
be the polar decomposition of $\left(\frac{h-t}{h}\right) R_0 + \left(\frac{t}{h}\right) R_1$, where $Q(t)$ is orthogonal and $Y(t)$ is symmetric positive definite.  Then
\[
Q(t) = R_0 e^{tK} + O(h^3)
\]
for every $t \in [0,h/2) \cup (h/2,h]$.  When $t=h/2$, the equality $R(t) = R_0 e^{tK}$ holds exactly.
\end{lemma}
\begin{proof}
See~\cite{Gawlik2016}.
\end{proof}

\subsection{Embedding in the Space of Quaternions} \label{sec:quat}

If $n=3$, then instead of embedding $SO(3)$ in $\mathbb{R}^{3 \times 3}$, we may opt to identify $SO(3)$ with the set of elements of unit length in the space $\mathbb{H}$ of quaternions.  Considered as a vector space, $\mathbb{H} = \mathbb{R}^4$, so the unit quaternions constitute the 3-sphere $S^3$.  Every vector $u \in S^3$ can be written in the form
\[
u = \left(\cos\left(\frac{\theta}{2}\right), v_1\sin\left(\frac{\theta}{2}\right), v_2\sin\left(\frac{\theta}{2}\right), v_3\sin\left(\frac{\theta}{2}\right)\right)
\]
for some $\theta \in [0,\pi]$ and some unit vector $v = (v_1,v_2,v_3) \in \mathbb{R}^3$.  In the usual identification of quaternions with rotations, $u$ is identified with a rotation about the axis $v$ by an angle $\theta$.  Under this correspondence, multiplication in $SO(3)$ corresponds to multiplication in $\mathbb{H}$ according to the rule
\[\begin{split}
(u_1,u_2,u_3,u_4)(w_1,w_2,w_3,w_4) = (&u_1 w_1 - u_2 w_2 - u_3 w_3 - u_4 w_4, \\
&u_1 w_2 + u_2 w_1 + u_3 w_4 - u_4 w_3, \\
&u_1 w_3 - u_2 w_4 + u_3 w_1 + u_4 w_2, \\
&u_1 w_4 + u_2 w_3 - u_3 w_2 + u_4 w_1).
\end{split}\]
The closest point projection $\mathcal{P}_{S^3} : \mathbb{R}^4 \setminus \{0\} \rightarrow S^3$ is nothing more than normalization:
\[
\mathcal{P}_{S^3}(q) = \frac{q}{\|q\|},
\]
where $\|\cdot\|$ denotes the Euclidean norm.  The invariance of the Euclidean norm under rotations implies that
\begin{equation} \label{biinvariantquat}
\mathcal{P}_{S^3}(uqw) = u \left(\mathcal{P}_{S^3} (q) \right) w, \quad \forall u,w \in S^3,
\end{equation}
where we have used concatenation to denote quaternion multiplication.

The analogues of the interpolation in operators~(\ref{interp0}) and~(\ref{interp1}) are straightforward to write down.  For a function $u \in C(\Omega,S^3)$, the analogue of~(\ref{interp0}) is the interpolant
\begin{equation} \label{interp2}
\mathcal{I}_{h,S^3} \, u(x) = \mathcal{P}_{S^3} \left( \sum_i \varphi_i(x) u_i \right),
\end{equation}
where $\{\varphi_i\}_i$ is a basis of Lagrange polynomials of degree $\le r$ on a simplex $K \subseteq \mathcal{D}$ containing $x$, and $\{u_i\}_i$ are the values of $u$ at the corresponding nodes of $K$.  Equivalently, by Proposition~\ref{lemma:geodesicFE}, this interpolant is a geodesic finite element on $S^3$ with respect to the chordal metric $\mathrm{dist}(v,w)=\|v-w\|$:
\[
\mathcal{I}_{h,S^3} \, u(x) = \argmin_{w \in S^3} \sum_i \varphi_i(x) \|w-u_i\|^2.
\]
For a function $u \in C^1([0,T],S^3)$, the analogue of~(\ref{interp1}) reads
\begin{equation} \label{interp3}
\begin{split}
\mathcal{I}_{h,S^3} \, u(t) = \mathcal{P}_{S^3} \left( \sum_{i=0}^1 \phi_i\left(\frac{t-t_k}{t_{k+1}-t_k}\right) u(t_{k+i}) + (t_{k+1}-t_k) \psi_i\left(\frac{t-t_k}{t_{k+1}-t_k}\right) \dot{u}(t_{k+i}) \right),  \\ t \in [t_k,t_{k+1}], \, 0 \le k < N,&
\end{split}
\end{equation}
where $\phi_0,\phi_1,\psi_0,\psi_1$ are the Hermite cubic basis functions~(\ref{hermite1}-\ref{hermite2}).

By~(\ref{biinvariantquat}), both~(\ref{interp2}) and~(\ref{interp3}) are equivariant under rotations, and they enjoy the same approximation properties as~(\ref{interp0}) and~(\ref{interp1}), respectively, in view of Propositions~\ref{lemma:Linferror} and~\ref{lemma:W1inferror}.

\subsubsection{Derivatives of the Interpolant}

To differentiate the interpolants~(\ref{interp2}) and~(\ref{interp3}), it is enough to derive formulas for the derivatives of $\mathcal{P}_{S^3}(q(t))$, where $q$ is an $\mathbb{R}^4$-valued function of a single parameter $t$.  This is a trivial calculus exercise that can be done without regarding $q$ as a quaternion, but the result is more illuminating when expressed in the language of quaternions.  To do so, we introduce the following notation.  If $q = (q_1,q_2,q_3,q_4) \in \mathbb{H}$, we denote by $q^* = (q_1,-q_2,-q_3,-q_4)$  the conjugate of $q$ and by $q^2=qq$ the square of $q$.  We denote the real and imaginary parts of $q$ by
\[
\mathrm{Re}(q) = \frac{1}{2}(q+q^*) = (q_1,0,0,0)
\]
and
\[
\mathrm{Im}(q) = \frac{1}{2}(q-q^*) = (0,q_2,q_3,q_4)
\]
respectively.  If $q$ is nonzero, we denote the inverse of $q$ by $q^{-1} = q^*/\|q\|^2$.  For a real quaternion $q=(q_1,0,0,0)$ and a scalar $x$, we write $q^x=(q_1^x,0,0,0)$.  In this notation, we obtain the following formulas involving the first and second derivatives of $\mathcal{P}_{S^3}(q(t))$.

\begin{lemma}
If $q \in C^1(I,\mathbb{H})$ is nonzero on an interval $I \subseteq \mathbb{R}$ and $u(t) = \mathcal{P}_{S^3}(q(t)) = \frac{q(t)}{\|q(t)\|}$, then
\begin{equation} \label{uinvudot}
u(t)^{-1} \dot{u}(t) = \mathrm{Im}\left(q(t)^{-1}\dot{q}(t)\right)
\end{equation}
for every $t \in I$.  If furthermore $q \in C^2(I,\mathbb{H})$, then
\begin{equation} \label{ddtuinvudot}
\frac{d}{dt} \left( u(t)^{-1} \dot{u}(t) \right) = \mathrm{Im}\left(q(t)^{-1}\ddot{q}(t) - \left(q(t)^{-1}\dot{q}(t)\right)^2 \right)
\end{equation}
for every $t \in I$.
\end{lemma}
\begin{proof}
Since $q^*q=qq^*=(\|q\|^2,0,0,0)$, we may write $u = q (q^*q)^{-1/2}$ and differentiate to obtain
\begin{align*}
\dot{u} 
&= \dot{q} (q^*q)^{-1/2} - \frac{1}{2} q (\dot{q}^*q + q^* \dot{q}) (q^*q)^{-3/2} \\
&= \frac{1}{2} \left( \frac{\dot{q}}{\|q\|} - \frac{q\dot{q}^*q}{\|q\|^3} \right) \\
&= \frac{1}{2} \left( \frac{\dot{q}}{\|q\|} - \frac{q \dot{q}^* (q^{-1})^*}{\|q\|} \right).
\end{align*}
Multiplying by $u^{-1} = \|q\| q^{-1}$ proves~(\ref{uinvudot}).  
To prove~(\ref{ddtuinvudot}), differentiate~(\ref{uinvudot}) and use the fact that$\frac{d}{dt} q^{-1} = -q^{-1}\dot{q}q^{-1}$.
\end{proof}

Note that it is sometimes the case that $q(t)^{-1}\dot{q}(t)$ is imaginary for certain values of $t$.  This holds, for instance, at each node $t_k$, $k=0,1,\dots,N$, when $q(t)$ is the input to $\mathcal{P}_{S^3}$ in~(\ref{interp3}).   If this is the case, then~(\ref{ddtuinvudot}) reduces to
\[
\frac{d}{dt} \left( u(t)^{-1} \dot{u}(t) \right) = \mathrm{Im}\left(q(t)^{-1}\ddot{q}(t)\right),
\]
since $\mathrm{Im}(v^2)=0$ for any imaginary $v \in \mathbb{H}$.

\section{Minimum Acceleration Curves on the Special Orthogonal Group} \label{sec:minaccel}

In this section, we use the preceding theory to construct a numerical method for approximating \emph{minimum acceleration} curves on $SO(n)$.  Roughly speaking, a minimum acceleration curve on $SO(n)$ is a smooth map $R : [0,T] \rightarrow SO(n)$ which locally minimizes
\begin{equation} \label{torque}
\int_0^T \|\dot{\Omega}\|^2 \, dt
\end{equation}
subject to certain constraints, where $\Omega = R^T \dot{R}$, and $\|\cdot\|$ is a norm on $\mathfrak{so}(n)$ (which we will take equal to the Frobenius norm in what follows).  When $n=3$, the matrix $R(t)$ can be thought of as rotation matrix specifying the orientation of a rigid body at time $t$, so that $\Omega$ and $\dot{\Omega}$ correspond the angular velocity and angular acceleration, respectively, of the body in a body-fixed frame.  In a typical application, a sequence of target directions is given, and a minimum acceleration curve passing through the target directions at specified times $0=\tau_0 < \tau_1 < \dots < \tau_M = T$ is sought.  More explicitly, given a sequence of vectors $v_0,v_1,\dots,v_M \in \mathbb{R}^n$, the task is to find a continuously differentiable map $R : [0,T] \rightarrow SO(n)$ which locally minimizes~(\ref{torque}) and satisfies $R(0)=I$ and $R(\tau_j)v_0=v_j$, $j=1,2,\dots,M$.

To state the minimization problem precisely, let $H^2(0,T;\mathbb{R}^{n \times n})$ denote the space of $\mathbb{R}^{n \times n}$-valued functions on $(0,T)$ with square-integrable second derivatives.  By the Sobolev embedding theorem, these functions are continuously differentiable, and we may define
\[
\mathcal{V}([0,T],SO(n)) = \{R \in H^2(0,T; \mathbb{R}^{n \times n}) : R(t) \in SO(n) \, \forall t \in [0,T] \text{ and } R(0)=I\}.
\]
For each $R \in \mathcal{V}([0,T],SO(n))$, denote $\Omega = R^T \dot{R}$.  The minimization problem we seek to approximate numerically reads
\begin{subequations} \label{minaccel}
\begin{alignat}{3} 
&\minimize_{R \in \mathcal{V}([0,T],SO(n))} &&\int_0^T \|\dot{\Omega}\|^2 \, dt \\
&\text{subject to } && R(\tau_j)v_0 = v_j, \quad j=1,2,\dots,M.
\end{alignat}
\end{subequations}
We remark that other variants of the constraints are possible, such as constraints on the values of $R$ and/or $\Omega$ at specified times.  These are easy to enforce using simple modifications to the setup detailed below.

A discretization of this problem can be constructed by searching for a minimizer within a finite-dimensional subspace $\mathcal{V}_h([0,T],SO(n)) \subset \mathcal{V}([0,T],SO(n))$.  Since functions in $\mathcal{V}([0,T],SO(n))$ are continuously differentiable, the same must be true of functions in $\mathcal{V}_h([0,T],SO(n))$.   
To this end, we consider below the two $C^1$ interpolants constructed in Section~\ref{sec:SOn}: the interpolant~(\ref{interp1}), which makes use of the embedding of $SO(n)$ in $\mathbb{R}^{n \times n}$, and the interpolant~(\ref{interp3}), which makes use of the identification of $SO(3)$ with the set of unit quaternions.

\subsection{Discretization with Matrices}  \label{sec:minaccel_Rnxn}

We begin by using the $C^1$ interpolant~(\ref{interp1}) to discretize~(\ref{minaccel}).  Choose a partition $0 = t_0 < t_1 < \dots < t_N = T$ of the interval $[0,T]$, and let
\[
\mathcal{V}_h([0,T],SO(n)) = \left\{ \mathcal{I}_{h,SO(n)}R : R \in \mathcal{V}([0,T],SO(n)) \right\},
\]
where $\mathcal{I}_{h,SO(n)}R \in C^1([0,T],SO(n))$ denotes the interpolant~(\ref{interp1}) detailed in Section~\ref{sec:Rnxn}.
Elements of $\mathcal{V}_h([0,T],SO(n))$ are functions $R : [0,T] \rightarrow SO(n)$ whose restrictions to each interval $[t_k,t_{k+1}]$ have the form
\begin{equation} \label{Roft}
R(t) = \mathcal{P}_{SO(n)} \left( \sum_{i=0}^1 \phi_i\left(\frac{t-t_k}{t_{k+1}-t_k}\right) R_{k+i} + (t_{k+1}-t_k) \psi_i\left(\frac{t-t_k}{t_{k+1}-t_k}\right) R_{k+i} \Omega_{k+i} \right),
\end{equation}
where $\{R_k\}_{k=0}^N \subset SO(n)$, $\{\Omega_k\}_{k=0}^N \subset \mathfrak{so}(n)$, and $\phi_i : [0,1] \rightarrow \mathbb{R}$ and $\psi_i : [0,1] \rightarrow \mathbb{R}$ are the scalar-valued Hermite cubic polynomials~(\ref{hermite1}-\ref{hermite2}). 
Note that for each $k$, the values of $R(t)$ and $\dot{R}(t)$ at $t=t_k$ are related to $R_k$ and $\Omega_k$ via
\begin{align*}
R(t_k) &= R_k, \\
\dot{R}(t_k) &= R_k\Omega_k.
\end{align*}
The discretization of~(\ref{minaccel}) reads
\begin{subequations} \label{minaccel_h}
\begin{alignat}{3} 
&\minimize_{R \in \mathcal{V}_h([0,T],SO(n))} &&\int_0^T \|\dot{\Omega}\|^2 \, dt \label{minaccel_ha} \\
&\text{subject to } && R(\tau_j)v_0 = v_j, \quad j=1,2,\dots,M, \label{minaccel_hb}
\end{alignat}
\end{subequations}
where, as before, $\Omega = R^T \dot{R}$.  Upon approximating the integral in~(\ref{minaccel_ha}) with quadrature, the problem~(\ref{minaccel_h}) is a constrained minimization problem in the unknowns $\{R_k\}_{k=0}^N \subset SO(n)$ and $\{\Omega_k\}_{k=0}^N \subset \mathfrak{so}(n)$.
Below we show that it can be recast as an unconstrained least-squares problem, thereby admitting a relatively efficient solution.  

For simplicity, let $n=3$ and assume that the partition $0 = t_0 < t_1 < \dots < t_N = T$ has been chosen in such a way that the set of target direction times $\{\tau_j\}_{j=1}^M$ is a subset of $\{t_k\}_{k=0}^N$.  That is, for each $j=0,1,\dots,M$, there exists $k_j \in \{0,1,\dots,N\}$ such that
\[
t_{k_j} = \tau_j.
\]
Fix a sequence $\{\bar{R}_k\}_{k=0}^N \subset SO(3)$ satisfying $\bar{R}_0 = I$ and 
\[
\bar{R}_{k_j} v_0 = v_j, \quad j=1,2,\dots,M.
\]
A natural choice is to set $\bar{R}_0 = I$ and define the sequence inductively by setting
\[
\bar{R}_{k_j+i} = \exp\left( \frac{i}{k_{j+1}-k_j} \widehat{a}_j  \right) \bar{R}_{k_j }, \quad i=1,2,\dots,k_{j+1}-k_j, \; j=0,1,\dots,M-1,
\]
where $a_j = v_j \times v_{j+1}$ and $\widehat{\cdot} : \mathbb{R}^3 \rightarrow \mathfrak{so}(3)$ denotes the ``hat map''
\[
\widehat{u} = \begin{pmatrix} 0 & -u_3 & u_2 \\ u_3 & 0 & -u_1 \\ -u_2 & u_1 & 0 \end{pmatrix}.
\]
Relative to this reference sequence $\{\bar{R}_k\}_{k=0}^N$, we can parametrize the sequence $\{R_k\}_{k=0}^N$ with variables $\{b_k\}_{k=0}^N \subset \mathbb{R}^3$ given by
\begin{equation} \label{RkRkbar}
\widehat{b}_k = \log(\bar{R}_k^T R_k) \iff R_k = \bar{R}_k \exp(\widehat{b}_k).
\end{equation}
These variables measure the deviation of $\{R_k\}_{k=0}^N$ from the reference sequence $\{\bar{R}_k\}_{k=0}^N$.  They offer two practical advantages.  First, they belong to a linear space, and second, they render the constraint~(\ref{minaccel_hb}) trivial to enforce.  Indeed,~(\ref{minaccel_hb}) holds for a given $j$ if and only if $b_{k_j} = \beta_{k_j} v_0$ for some scalar $\beta_{k_j}$.  In words, $R_{k_j}$ must differ from $\bar{R}_{k_j}$ (if at all) by a rotation about the axis $v_0$.  

In a similar manner, we can parametrize each unknown $\Omega_k \in \mathfrak{so}(3)$ with its preimage $\omega_k \in \mathbb{R}^3$ under the hat map, i.e.
\begin{equation} \label{Omegak}
\Omega_k = \widehat{\omega}_k.
\end{equation}
If we denote
\begin{align*}
\mathcal{K} &= \{k_1, k_2, \dots, k_M\}, \\
\mathcal{K}^c &= \{1,2,\dots,N\} \setminus \mathcal{K},
\end{align*}
then the problem~(\ref{minaccel}) reduces to an unconstrained minimization problem in the unknowns $\{b_k\}_{k \in \mathcal{K}^c} \subset \mathbb{R}^3$, $\{\beta_k\}_{k \in \mathcal{K}} \subset \mathbb{R}$, and $\{\omega_k\}_{k=0}^N \subset \mathbb{R}^3$.  Upon approximating the integral in~(\ref{minaccel_ha})  with a quadrature rule of the form
\begin{equation} \label{quadrature}
\int_0^T \|\dot{\Omega}\|^2 \, dt \approx \sum_{i=1}^{N_q} w_i \|\dot{\Omega}(s_i)\|^2
\end{equation}
with weights $w_i>0$ and nodes $s_i \in [0,T]$, $i=1,2,\dots,N_q$, this minimization problem reads
\begin{alignat}{3} \label{lsq}
\minimize_{x} \;&& g(x)^T g(x),
\end{alignat}
where $x \in \mathbb{R}^{6N-2M+3}$ is a vector containing the unknowns $\{b_k\}_{k \in \mathcal{K}^c} \subset \mathbb{R}^3$, $\{\beta_k\}_{k \in \mathcal{K}} \subset \mathbb{R}$, and $\{\omega_k\}_{k=0}^N \subset \mathbb{R}^3$; $g(x) \in \mathbb{R}^{3N_q}$ is a vector with components
\begin{equation} \label{residualmat}
g_{3i+j}(x) = \sqrt{w_i} \alpha_{ij}, \quad i=1,2,\dots,N_q, \; j=1,2,3;
\end{equation}
and $\alpha_{i1},\alpha_{i2},\alpha_{i3}$ are the 3 independent components of $\dot{\Omega}(s_i)$.
Here, of course, $\dot{\Omega}$ is obtained from the vector of unknowns $x$ by evaluating~(\ref{RkRkbar}) and~(\ref{Omegak}) to recover $R_k$ and $\Omega_k$, substituting into~(\ref{Roft}), and differentiating $\Omega=R^T \dot{R}$.  For further details on evaluating~(\ref{residualmat}) (as well as its Jacobian) see Section~\ref{sec:algorithm}.

\subsection{Discretization with Quaternions} \label{sec:minaccel_quat}

Assuming still that $n=3$, there is a second way to discretize the optimization problem~(\ref{minaccel}).  Instead of using the interpolant~(\ref{interp1}) to approximate functions in $\mathcal{V}([0,T],SO(3))$, we may use the interpolant~(\ref{interp3}), which takes advantage of the identification of $SO(3)$ with the set of unit quaternions.  

More precisely, let $\{t_k\}_{k=0}^N$, $\{\tau_j\}_{j=0}^M$, $\{v_j\}_{j=0}^M$, $\mathcal{K}$, and $\mathcal{K}^c$ be as in Section~\ref{sec:minaccel_Rnxn}.  Define
\[
\mathcal{V}([0,T],S^3) = \{ u \in H^2([0,T],\mathbb{R}^4) : u(t) \in S^3 \, \forall t \in [0,T] \text{ and } u(0) = (1,0,0,0) \}
\]
and
\[
\mathcal{V}_h([0,T],S^3) = \left\{ \mathcal{I}_{h,S^3}u : u \in \mathcal{V}([0,T],S^3) \right\},
\]
where $\mathcal{I}_{h,S^3}u \in C^1([0,T],S^3)$ denotes the interpolant~(\ref{interp3}) detailed in Section~\ref{sec:quat}.  Elements of $\mathcal{V}_h([0,T],S^3)$ are functions $u : [0,T] \rightarrow S^3$ whose restrictions to each interval $[t_k,t_{k+1}]$ have the form
\begin{equation} \label{uoft}
u(t) = \mathcal{P}_{S^3} \left( \sum_{i=0}^1 \phi_i\left(\frac{t-t_k}{t_{k+1}-t_k}\right) u_{k+i} + (t_{k+1}-t_k) \psi_i\left(\frac{t-t_k}{t_{k+1}-t_k}\right) u_{k+i} (0,\omega_{k+i}) \right),
\end{equation}
where $\{u_k\}_{k=0}^N \subset S^3$, $\{\omega_k\}_{k=0}^N \subset \mathbb{R}^3$, and $\phi_i : [0,1] \rightarrow \mathbb{R}$ and $\psi_i : [0,1] \rightarrow \mathbb{R}$ are the scalar-valued Hermite cubic polynomials~(\ref{hermite1}-\ref{hermite2}).  Here, $\mathcal{P}_{S^3}(q) = q/\|q\|$, and $u_{k+i} (0,\omega_{k+i})$ denotes the product of two quaternions $u_{k+i}$ and $(0,\omega_{k+i})$. 

With this choice of finite-dimensional function space, the discretization of~(\ref{minaccel}) reads
\begin{subequations} \label{minaccel_quat}
\begin{alignat}{3} 
&\minimize_{u \in \mathcal{V}_h([0,T],S^3)} &&\int_0^T \|\dot{\omega}\|^2 \, dt \label{minaccel_quata} \\
&\text{subject to } && u(\tau_j) \cdot v_0 = v_j, \quad j=1,2,\dots,M, \label{minaccel_quatb}
\end{alignat}
\end{subequations}
where $(0,\omega(t)) = u(t)^{-1} \dot{u}(t)$, and $q \cdot v$ denotes the action of a unit quaternion $q$ on a vector $v \in \mathbb{R}^3$:
\[
 (0,q \cdot v) = q (0,v)q^{-1}.
\] 
The action so defined realizes a rotation of $v$ by $q$ under the usual identification of quaternions with rotations.

As in Section~\ref{sec:minaccel_Rnxn}, the problem~(\ref{minaccel_quat}) can be recast as an unconstrained least squares problem.  To do so, fix a reference sequence $\{\bar{u}_k\}_{k=0}^N$ that satisfies
\[
\bar{u}_{k_j} \cdot v_0 = v_j, \quad j=1,2,\dots,M.
\]
Define $b_k \in \mathbb{R}^3$ via
\begin{equation} \label{ukukbar}
(0,b_k) = \log(\bar{u}_k^* u_k) \iff u_k = \bar{u}_k \exp(0,b_k),
\end{equation}
where
\[
\exp(0,v) = \left( \cos\left( \frac{\|v\|}{2} \right), \frac{v}{\|v\|} \sin\left( \frac{\|v\|}{2} \right) \right)
\]
denotes the quaternion exponential, and $\log$ denotes its inverse.  The constraint~(\ref{minaccel_quatb}) then reduces to the requirement that for each $j$, $b_{k_j} = \beta_{k_j} v_0$ for some scalar $\beta_{k_j}$.  

It follows that, after approximating~(\ref{minaccel_quata}) with a quadrature rule of the form~(\ref{quadrature}), the problem~(\ref{minaccel_quat}) can be written in the form~(\ref{lsq}).  In this formulation, $x \in \mathbb{R}^{6N-2M+3}$ is a vector containing the unknowns $\{b_k\}_{k \in \mathcal{K}^c} \subset \mathbb{R}^3$, $\{\beta_k\}_{k \in \mathcal{K}} \subset \mathbb{R}$, and $\{\omega_k\}_{k=0}^N \subset \mathbb{R}^3$; $g(x) \in \mathbb{R}^{3N_q}$ is a vector with components
\begin{equation} \label{residualquat}
g_{3i+j}(x) = \sqrt{w_i} \alpha_{ij}, \quad i=1,2,\dots,N_q, \; j=1,2,3;
\end{equation}
$\alpha_{i1},\alpha_{i2},\alpha_{i3}$ are the 3 nonzero components of $\dot{\omega}(s_i)$; and $\omega(t) = u(t)^{-1} \dot{u}(t)$.   A detailed algorithm for evaluating~(\ref{residualquat}) and its Jacobian is given in Section~\ref{sec:algorithm}.

\subsection{Numerical Examples}

To solve the least-squares problem~(\ref{lsq}), we consider here the Levenberg-Marquardt algorithm~\cite{more1978levenberg}, which computes a solution via the iteration 
\begin{equation} \label{lmalg}
x^{(m+1)} = x^{(m)} + \left(J(x^{(m)})^T J(x^{(m)}) + \lambda I\right)^{-1} J(x^{(m)})^T g(x^{(m)}),
\end{equation}
starting from an initial guess $x^{(0)}$.  Here, $J(x^{(m)}) \in \mathbb{R}^{3N_q \times (6N-2M+3)}$ denotes the Jacobian of $g$ at $x=x^{(m)}$, and $\lambda$ is a parameter chosen heuristically.  
In our numerical experiments, we initiated $\lambda=0.01$ and updated its value at each iteration via the following heuristic: If the update~(\ref{lmalg}) leads to a decrease in the objective function $g(x)^T g(x)$, accept the update and decrease $\lambda$ by a factor of 10; otherwise, reject the update and increase $\lambda$ by a factor of 10. 

\begin{table}[t]
\centering
\begin{filecontents*}{Data/materr.dat}
   8.0000000e+00   8.5185221e-04   4.2628436e-03
   1.6000000e+01   4.0751099e-05   2.9027670e-04
   3.2000000e+01   2.3595191e-06   2.8984848e-05
   6.4000000e+01   1.5455551e-07   3.4306141e-06
\end{filecontents*}
\pgfplotstabletypeset[
every head row/.style={after row=\midrule},
create on use/rate1/.style={create col/dyadic refinement rate={1}},
create on use/rate2/.style={create col/dyadic refinement rate={2}},
columns={0,1,rate1,2,rate2},
columns/0/.style={sci zerofill,column type/.add={}{|},column name={$N$}},
columns/1/.style={sci zerofill,precision=3,column type/.add={}{|},column name={$L^2$-error}},
columns/2/.style={sci zerofill,precision=3,column type/.add={}{|},column name={$H^1$-error}}, 
columns/rate1/.style={fixed zerofill,precision=3,column type/.add={}{|},column name={Order}},
columns/rate2/.style={fixed zerofill,precision=3,column name={Order}}
]
{Data/materr.dat}
\caption{Error in the numerical solution of~(\ref{minaccel}) on the interval $[0,1]$ with target directions given by~(\ref{example1vectors}), obtained using the matrix-based discretization detailed in Section~\ref{sec:minaccel_Rnxn}.  The solution was computed on a uniform partition of $[0,1]$ into $N$ intervals of equal length.}
\label{tab:matconvergence}
\end{table}

\begin{table}[t]
\centering
\begin{filecontents*}{Data/quaterr.dat}
   4.0000000e+00   2.4355981e-04   1.4362613e-03
   8.0000000e+00   1.4673543e-05   1.5760920e-04
   1.6000000e+01   8.9546938e-07   1.8954626e-05
   3.2000000e+01   5.6713673e-08   2.3529418e-06
\end{filecontents*}
\pgfplotstabletypeset[
every head row/.style={after row=\midrule},
create on use/rate1/.style={create col/dyadic refinement rate={1}},
create on use/rate2/.style={create col/dyadic refinement rate={2}},
columns={0,1,rate1,2,rate2},
columns/0/.style={sci zerofill,column type/.add={}{|},column name={$N$}},
columns/1/.style={sci zerofill,precision=3,column type/.add={}{|},column name={$L^2$-error}},
columns/2/.style={sci zerofill,precision=3,column type/.add={}{|},column name={$H^1$-error}}, 
columns/rate1/.style={fixed zerofill,precision=3,column type/.add={}{|},column name={Order}},
columns/rate2/.style={fixed zerofill,precision=3,column name={Order}}
]
{Data/quaterr.dat}
\caption{Error in the numerical solution of~(\ref{minaccel}) on the interval $[0,1]$ with target directions given by~(\ref{example1vectors}), obtained using the quaternion-based discretization detailed in Section~\ref{sec:minaccel_quat}.  The solution was computed on a uniform partition of $[0,1]$ into $N$ intervals of equal length.}
\label{tab:quatconvergence}
\end{table}

\begin{figure}[t]
    \centering
    \begin{subfigure}[b]{0.3\textwidth}
        \includegraphics[width=\textwidth,trim=17.5in 2in 17.5in 1.9in,clip=true]{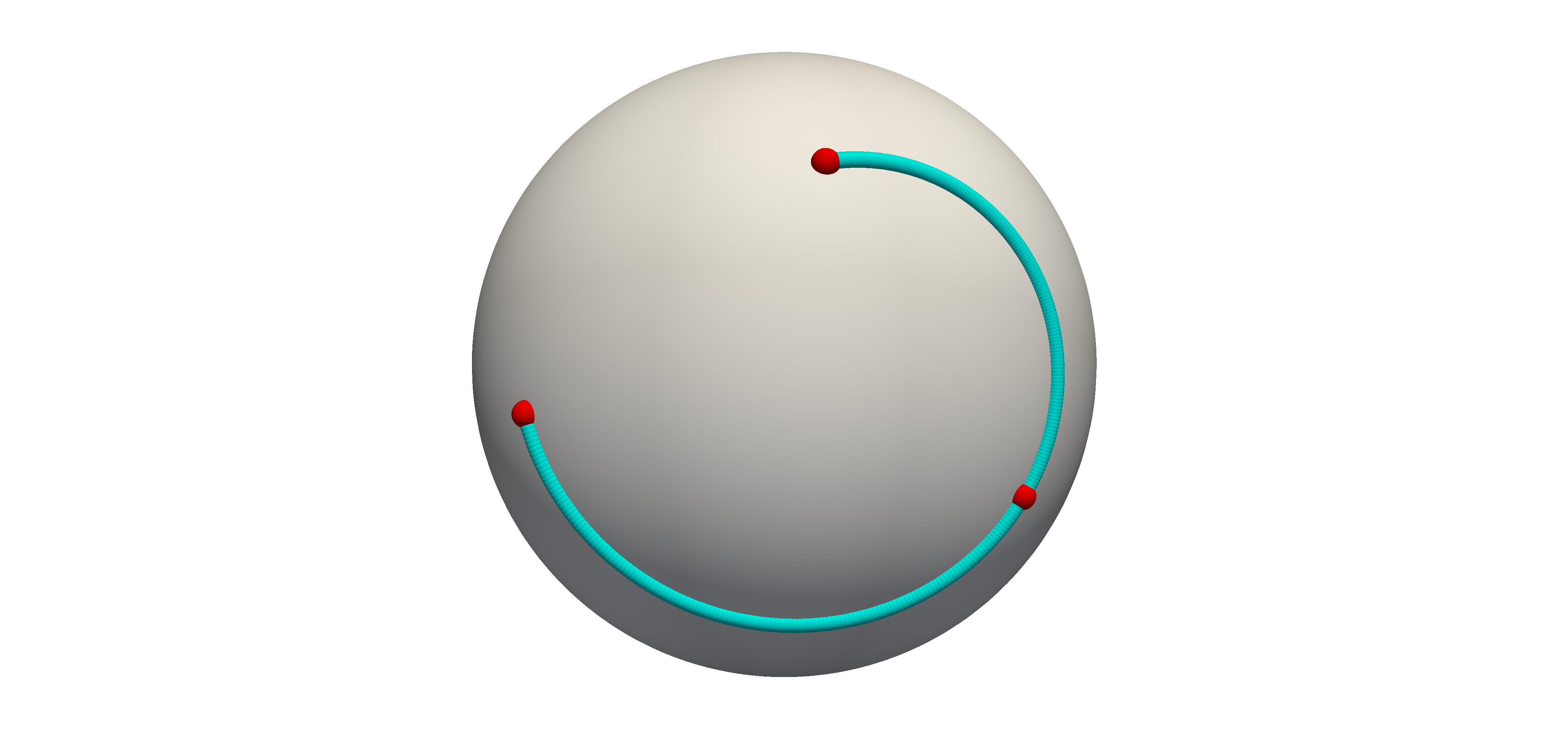}
        \caption{}
        \label{fig:S2curve}
    \end{subfigure}
    \begin{subfigure}[b]{0.3\textwidth}
        \includegraphics[width=\textwidth,trim=18.5in 1.6in 16.5in 1.8in,clip=true]{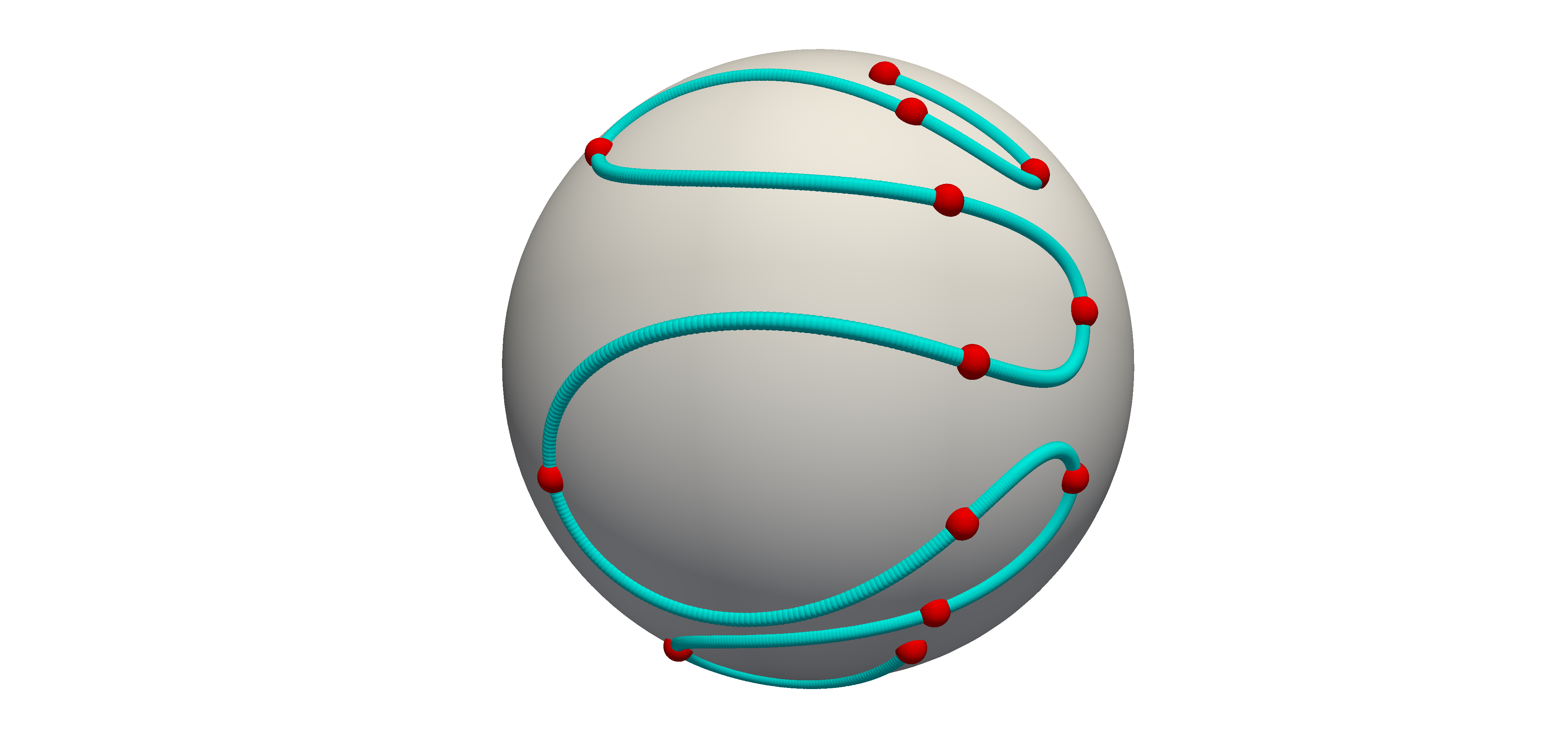}
        \caption{}
        \label{fig:slalom}
    \end{subfigure}
    \begin{subfigure}[b]{0.3\textwidth}
        \includegraphics[width=\textwidth,trim=17.5in 1.9in 17.5in 1.9in,clip=true]{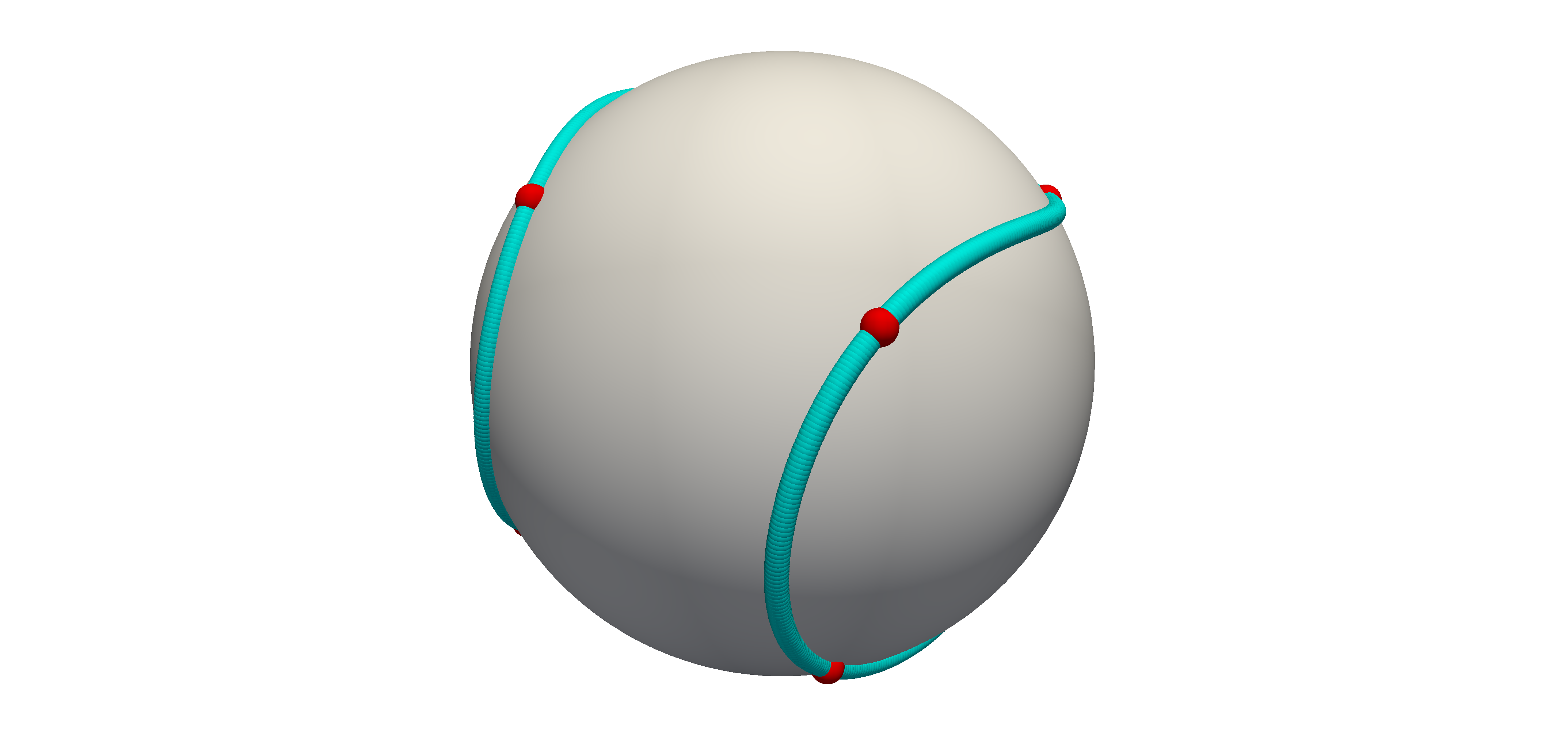}
        \caption{}
        \label{fig:bb}
    \end{subfigure}
\caption{Numerical solutions to the minimum acceleration problem~(\ref{minaccel}) with target directions $\{v_j\}_{j=0}^M \subset \mathbb{R}^3$ given by: (a) equation~(\ref{example1vectors}), (b) equation~(\ref{slalom}), and (c) equations~(\ref{bb1}-\ref{bb2}).  To visualize the minimum acceleration curves $R : [0,T] \rightarrow SO(3)$, we have plotted $R(t)v_0$, $t \in [0,T]$, on the unit sphere.  The target directions $\{v_j\}_{j=0}^M$ are marked in red.}
\end{figure}

We applied this algorithm to compute minimum acceleration curves on the interval $[0,1]$ with $M+1=3$ target directions (equally spaced in time) given by 
\begin{equation} \label{example1vectors}
v_0 = (1,0,0), \quad v_1 = \left(0,1,0\right), \quad v_2 = \left(\frac{1}{\sqrt{6}},\frac{1}{\sqrt{6}},\frac{2}{\sqrt{6}}\right).
\end{equation}
We solved the problem on a uniform partition $0=t_0<t_1<\dots<t_N=1$ of $[0,1]$ into $N$ intervals of equal length using two discretizations: the matrix-based discretization detailed in Section~\ref{sec:minaccel_Rnxn}, and the quaternion-based discretization detailed in Section~\ref{sec:minaccel_quat}.  Figure~\ref{fig:S2curve} shows a representative numerical solution to this problem (obtained with the quaternion-based discretization with $N=8$), which we have visualized by plotting $u(t) \cdot v_0$, $t \in [0,1]$, on the unit sphere.   Tables~\ref{tab:matconvergence} and~\ref{tab:quatconvergence} show the errors between the computed solutions and the exact solution as a function of $N$.  Since an exact solution is not known analytically, we approximated it using a refined discretization ($N=1024$).  The errors reported in the table for the matrix-based discretization are the $L^2$-error
\[
\left( \int_0^1 \|R(t)-R_{exact}(t)\|^2 \, dt \right)^{1/2}
\]
between the approximate solution $R(t) \in SO(3)$ and the exact solution $R_{exact}(t) \in SO(3)$, and the $H^1$-error
\[
\left( \int_0^1 \|\dot{R}(t)-\dot{R}_{exact}(t)\|^2 \, dt \right)^{1/2}.
\]
Similarly, for the quaternion-based discretization, the errors reported are the $L^2$-error
\[
\left( \int_0^1 \|u(t)-u_{exact}(t)\|^2 \, dt \right)^{1/2}
\]
and the $H^1$-error
\[
\left( \int_0^1 \|\dot{u}(t)-\dot{u}_{exact}(t)\|^2 \, \right)^{1/2}
\]
between the approximate solution $u(t) \in S^3$ and the exact solution $u_{exact}(t) \in S^3$.  All integrals were computed using 4-point Gaussian quadrature elementwise.

The results in Tables~\ref{tab:matconvergence} and~\ref{tab:quatconvergence} indicate that both discretizations achieve optimal rates of convergence under refinement.  Namely, the discretizations converge with order 4 in the $L^2$-norm and with order 3 in the $H^1$-norm, consistent with the theoretical interpolation accuracy of the interpolants~(\ref{interp1}) and~(\ref{interp3}).  However, the quaternion-based discretization outperforms the matrix-based discretization in an absolute sense.  For each $N$, the error committed by the quaternion-based discretization is between one and two orders of magnitude smaller than that committed by the matrix-based discretization.  We also observed that the matrix-based discretization requires larger $N$ before the the asymptotic convergence rates are realized, which is why we have reported errors for larger values of $N$ in Table~\ref{tab:matconvergence} than in Table~\ref{tab:quatconvergence}.  The inferiority of the matrix-based discretization is compounded by the fact that, for fixed $N$, it requires more computational effort to evaluate the interpolant and its derivatives than does the quaternion-based discretization.  Indeed, to evaluate the interpolant, the former requires computing the polar decomposition of a matrix, while the latter requires normalizing a vector, a decidedly less expensive task.

\subsubsection*{Other Examples}  
For illustrative purposes, we have numerically computed and plotted in Figures~\ref{fig:slalom} and~\ref{fig:bb} two other minimum acceleration curves.  Figure~\ref{fig:slalom} shows a numerical solution for the case in which $M=12$ and
\begin{equation} \label{slalom}
v_j = \mathcal{P}_{S^2} \left( \frac{1}{2}+\frac{4}{5}\cos\left( \frac{\pi j}{2} \right), \, \frac{1}{2}, \, 1-\frac{j}{6} \right), \quad \tau_j = \frac{j}{12}, \quad  j=0,1,\dots,12,
\end{equation}
where $\mathcal{P}_{S^2}(w) = \frac{w}{\|w\|}$ for each nonzero $w \in \mathbb{R}^3$.  
Figure~\ref{fig:bb} shows a numerical solution for the case in which $M=8$, $\tau_j = \frac{j}{8}$ for each $j$, and
\begin{align}
v_0 &= v_8 = \frac{(1,1,1)}{\sqrt{3}}, &\quad v_1 &= \frac{(-1,1,1)}{\sqrt{3}}, &\quad v_2 &= \frac{(-1,-1,1)}{\sqrt{3}}, &\quad v_3 &= \frac{(1,-1,1)}{\sqrt{3}}, \label{bb1} \\
v_4 &= \frac{(1,-1,-1)}{\sqrt{3}}, &\quad v_5 &= \frac{(-1,-1,-1)}{\sqrt{3}}, &\quad v_6 &= \frac{(-1,1,-1)}{\sqrt{3}}, &\quad v_7 &= \frac{(1,1,-1)}{\sqrt{3}}. \label{bb2}
\end{align}
In the latter example, we imposed periodicity on the solution by introducing the constraints $u(1)=u(0)$ and $\dot{u}(1)=\dot{u}(0)$.  In both examples, we used the quaternion-based discretization on a partition of $[0,1]$ into $N=M$ intervals of equal length.

\subsection{Algorithmic Details} \label{sec:algorithm}

In this section, we detail a pair of algorithms for evaluating the residual vectors~(\ref{residualmat}) and~(\ref{residualquat}) and their Jacobians.

In order to specify the Jacobians of~(\ref{residualmat}) and~(\ref{residualquat}), it is necessary to specify an ordering of the variables $\{b_k\}_{k \in \mathcal{K}^c} \subset \mathbb{R}^3$, $\{\beta_k\}_{k \in \mathcal{K}} \subset \mathbb{R}$, and $\{\omega_k\}_{k=0}^N \subset \mathbb{R}^3$ that constitute the vector $x$ on which the residual vectors depend. We assume that the variables are ordered as
\begin{equation} \label{ordering}
x = (\omega_0,y_1,\omega_1,y_2,\omega_2,\dots,y_N,\omega_N)
\end{equation}
where $y_k = \beta_k$ if $k \in \mathcal{K}$ and $y_k = b_k$ if $k \in \mathcal{K}^c$.  Note that the dimension of $y_k$ (which we denote by $n_k$ in what follows) varies with $k$.  This fact slightly complicates the indexing of variables in the algorithms that follow, but the ordering~(\ref{ordering}) endows the Jacobian with an appealing sparsity pattern.  We also assume that a $P$-point quadrature rule is adopted elementwise, so that $N_q=NP$ and the integral of $\|\dot{\Omega}\|^2$ is approximated as
\begin{equation} \label{elementwisequadrature}
\int_0^T \|\dot{\Omega}\|^2 \, dt \approx \sum_{k=0}^{N-1} \sum_{p=1}^P (t_{k+1}-t_k) W_p \|\dot{\Omega}((1-\xi_p)t_k+\xi_p t_{k+1})\|^2
\end{equation}
for some quadrature weights $\{W_p\}_{p=1}^P \subset \mathbb{R}$ and nodes $\{\xi_p\}_{p=1}^P \subset [0,1]$ designed for integration on the unit interval.

We now state the algorithms, beginning with the matrix-based discretization.

\vspace{1em}
\begin{myenv}
\hrule\vspace{0.2em}
\captionof{mytype}{This algorithm evaluates the residual vector~(\ref{residualmat}) and its Jacobian $J$ for the matrix-based discretization detailed in Section~\ref{sec:Rnxn}.}
\vspace{-1.2em}\hrule\vspace{0.2em}
\begin{algorithmic}[1]
\Require Variables $\{b_k\}_{k \in \mathcal{K}^c} \subset \mathbb{R}^3$, $\{\beta_k\}_{k \in \mathcal{K}} \subset \mathbb{R}$, and $\{\omega_k\}_{k=0}^N \subset \mathbb{R}^3$; reference sequence $\{\bar{R}_k\}_{k=0}^N \subset SO(3)$; quadrature weights $\{W_p\}_{p=1}^P \subset \mathbb{R}$ and nodes $\{\xi_p\}_{p=1}^P \subset [0,1]$; vector $v_0 \in \mathbb{R}^3$
\Ensure Residual vector $g$ and its Jacobian $J$
\State $g = 0$ (size $3NP \times 1$)
\State $J = 0$ (size $3NP \times (6N-2M+3)$, where $M=|\mathcal{K}|$)
\State $e_1 = (1,0,0)$, $e_2 = (0,1,0)$, $e_3 = (0,0,1)$
\State $n_0 = 0$
\State $R_0  = \bar{R}_0$
\For{$k=1,2,\dots,N$} \Comment{\parbox[t]{.68\linewidth}{Calculate $\{R_k\}_{k=1}^N$ from $\{\bar{R}_k\}_{k=1}^N$, $\{b_k\}_{k \in \mathcal{K}^c}$, $\{\beta_k\}_{k \in \mathcal{K}}$.}}
\If{$k \in \mathcal{K}$}
\State $n_k = 1$
\State \label{algexp1} $R_k = \bar{R}_k \exp(\beta_k \widehat{v}_0)$
\State $\frac{\partial R_k}{\partial \beta_k} = R_k \widehat{v}_0$
\Else
\State $n_k = 3$
\State \label{algexp2} $R_k = \bar{R}_k \exp(\widehat{b}_k)$
\For{$j=1,2,3$}
\State \label{algdexp} $\frac{\partial R_k}{\partial b_{k,j}} = \bar{R}_k \mathrm{dexp}_{\widehat{b}_k} \widehat{e}_j$
\EndFor
\EndIf
\EndFor
\For{$k=0,1,\dots,N-1$} \Comment{\parbox[t]{.4\linewidth}{Loop over elements $(t_k,t_{k+1})$.}}
\State $h = t_{k+1} - t_k$
\For{$p=1,2,\dots,P$} \Comment{\parbox[t]{.4\linewidth}{Loop over quadrature points $\xi_p$.}}
\State $A = \hspace{1.8em} \sum_{i=0}^1 \Big( \phi_i(\xi_p) R_{k+i} + h \psi_i(\xi_p) R_{k+i} \widehat{\omega}_{k+i} \Big)$
\State $\dot{A} = h^{-1} \sum_{i=0}^1 \left( \dot{\phi}_i(\xi_p) R_{k+i} + h\dot{\psi}_i(\xi_p) R_{k+i} \widehat{\omega}_{k+i} \right)$
\State $\ddot{A} = h^{-2} \sum_{i=0}^1  \left( \ddot{\phi}_i(\xi_p) R_{k+i} + h \ddot{\psi}_i(\xi_p) R_{k+i} \widehat{\omega}_{k+i} \right)$
\State Use the iteration~(\ref{Xupdate2}-\ref{Tupdate2}) to calculate $T_\infty$. \Comment{\parbox[t]{.3\linewidth}{$T_\infty = R^T \ddot{R}$, where}}
\State \Comment{\parbox[t]{.3\linewidth}{$R(t) = \mathcal{P}_{SO(3)} A(t)$.}}
\State \label{algalphahat} $\widehat{\alpha} = \mathrm{skew}(T_\infty)$
\For{$j=1,2,3$}
\State $g_{3(Pk+p-1)+j} = \sqrt{hW_p} \alpha_j$
\EndFor
\For{$i=0,1$} \Comment{\parbox[t]{.48\linewidth}{Loop over variables on which $\alpha$}}
\For{$\ell = 1,2,\dots,n_{k+i}+3$} \Comment{\parbox[t]{.48\linewidth}{depends, namely, $\{y_k,\omega_k,y_{k+1},\omega_{k+1}\}$,}}
\If{$\ell \le n_{k+i}$} \Comment{\parbox[t]{.48\linewidth}{where  $y_k \in \{b_k,\beta_k\}$ and}}
\State \Comment{\parbox[t]{.48\linewidth}{$y_{k+1} \in \{b_{k+1},\beta_{k+1}\}$.}}
\If{$n_{k+i}=1$}
\State $\delta R_{k+i} = \frac{\partial R_{k+i}}{\partial \beta_{k+i}}$
\Else
\State $\delta R_{k+i} = \frac{\partial R_{k+i}}{\partial b_{k+i,\ell}}$
\EndIf
\State $\delta A = \hspace{2.3em} \phi_i(\xi_p) \delta R_{k+i} + h \psi_i(\xi_p) \delta R_{k+i} \widehat{\omega}_{k+i}$
\State $\delta \dot{A} = h^{-1} \left( \dot{\phi}_i(\xi_p) \delta R_{k+i} + h\dot{\psi}_i(\xi_p) \delta R_{k+i} \widehat{\omega}_{k+i} \right)$
\State $\delta \ddot{A} =  h^{-2} \left( \ddot{\phi}_i(\xi_p) \delta R_{k+i} + h\ddot{\psi}_i(\xi_p) \delta R_{k+i} \widehat{\omega}_{k+i} \right)$
\Else
\State $\delta\omega_{k+i} = e_{\ell-n_{k+i}}$
\State $\delta A = \hspace{1.1em} h \psi_i(\xi_p) R_{k+i} \widehat{\delta\omega}_{k+i}$
\State $\delta \dot{A} = \hspace{1.6em}  \dot{\psi}_i(\xi_p) R_{k+i} \widehat{\delta\omega}_{k+i}$
\State $\delta \ddot{A} = h^{-1} \ddot{\psi}_i(\xi_p) R_{k+i} \widehat{\delta\omega}_{k+i}$
\EndIf
\State Use the iteration~(\ref{Xupdate2}-\ref{Zupdate2}) \Comment{\parbox[t]{.4\linewidth}{$U_\infty = R^T \delta R$ and $Z_\infty = R^T \delta \ddot{R}$,}}
\State to calculate $U_\infty$ and $Z_\infty$. \Comment{\parbox[t]{.4\linewidth}{where $R(t) = \mathcal{P}_{SO(3)} A(t)$.}}
\State $\widehat{\delta\alpha} = \mathrm{skew}\left( U_\infty^T T_\infty + Z_\infty \right)$ \Comment{\parbox[t]{.4\linewidth}{Derivative of  $\widehat{\alpha}$ w.r.t.}} 
\State \Comment{\parbox[t]{.4\linewidth}{current variable.}}
\State $m = \ell + \sum_{r=0}^{k+i-1} (n_r+3)$ \Comment{\parbox[t]{.4\linewidth}{Index of current variable.}}
\For{$j=1,2,3$}
\State $J_{3(Pk+p-1)+j,m} = \sqrt{hW_p} \, \delta \alpha_j$
\EndFor
\EndFor
\EndFor
\EndFor
\EndFor
\end{algorithmic}
\end{myenv}
\hrule

\vspace{1em}

Note that in line~\ref{algalphahat} of the preceding algorithm, we have made use of the fact that $\Omega = R^T \dot{R}$ is skew-symmetric, so
\begin{align*}
\dot{\Omega} 
&= \frac{d}{dt} \mathrm{skew}(R^T \dot{R}) \\
&= \mathrm{skew}(R^T \ddot{R} + \dot{R}^T \dot{R}) \\
&= \mathrm{skew}(R^T \ddot{R}).
\end{align*}

Note also that lines~\ref{algexp1},~\ref{algexp2}, and~\ref{algdexp} of the preceding algorithm require the computation of the exponential $\exp : \mathfrak{so}(3) \rightarrow SO(3)$ and its derivative $\mathrm{dexp} : \mathfrak{so}(3) \times \mathfrak{so}(3) \rightarrow SO(3)$.  These maps can be computed explicitly with Rodrigues' formula
\[
\exp(\widehat{w}) = I + \frac{\sin\|w\|}{\|w\|} \widehat{w}   + \frac{1-\cos(\|w\|)}{\|w\|^2} \widehat{w}^2 
\]
and the formula~\cite{gallego2015compact}
\[
\mathrm{dexp}_{\widehat{w}} \widehat{v} = \frac{ (w \cdot v) \widehat{w} + \left( w \times (I-\exp(\widehat{w}))v\right)^{\widehat{\hspace{0.8em}}} }{\|w\|^2} \exp(\widehat{w}).
\]

Next, we detail the algorithm for the quaternion-based discretization.  Throughout the algorithm, we denote quaternion multiplication with concatenation.

\vspace{1em}
\begin{myenv}
\hrule\vspace{0.2em}
\captionof{mytype}{This algorithm evaluates the residual vector~(\ref{residualquat}) and its Jacobian for the quaternion-based discretization detailed in Section~\ref{sec:quat}.}
\vspace{-1.2em}\hrule\vspace{0.2em}
\begin{algorithmic}[1]
\Require Variables $\{b_k\}_{k \in \mathcal{K}^c} \subset \mathbb{R}^3$, $\{\beta_k\}_{k \in \mathcal{K}} \subset \mathbb{R}$, and $\{\omega_k\}_{k=0}^N \subset \mathbb{R}^3$; reference sequence $\{\bar{u}_k\}_{k=0}^N \subset S^3$; quadrature weights $\{W_p\}_{p=1}^P \subset \mathbb{R}$ and nodes $\{\xi_p\}_{p=1}^P \subset [0,1]$; vector $v_0 \in \mathbb{R}^3$
\Ensure Residual vector $g$ and its Jacobian $J$
\State $g = 0$ (size $3NP \times 1$)
\State $J = 0$ (size $3NP \times (6N-2M+3)$, where $M=|\mathcal{K}|$)
\State $e_1 = (1,0,0)$, $e_2 = (0,1,0)$, $e_3 = (0,0,1)$
\State $u_0  = \bar{u}_0$
\For{$k=1,2\dots,N$} \Comment{\parbox[t]{.68\linewidth}{Calculate $\{u_k\}_{k=1}^N$ from $\{\bar{u}_k\}_{k=1}^N$, $\{b_k\}_{k \in \mathcal{K}^c}$, $\{\beta_k\}_{k \in \mathcal{K}}$.}}
\If{$k \in \mathcal{K}$}
\State $n_k = 1$
\State $u_k = \bar{u}_k \left( \cos\left(\frac{\beta_k}{2}\right), v_0 \sin\left(\frac{\beta_k}{2}\right)  \right)$
\State $\frac{\partial u_k}{\partial \beta_k} = \frac{1}{2} \bar{u}_k \left( -\sin\left(\frac{\beta_k}{2}\right), v_0 \cos\left(\frac{\beta_k}{2}\right)  \right)$
\Else
\State $n_k = 3$
\State $u_k = \bar{u}_k \left( \cos\left(\frac{\|b_k\|}{2}\right), \frac{b_k}{\|b_k\|} \sin\left(\frac{\|b_k\|}{2}\right)  \right)$
\For{$j=1,2,3$}
\State $\frac{\partial u_k}{\partial b_{k,j}} = \bar{u}_k \Big($ \parbox[t]{.6\linewidth}{$-\frac{b_{k,j}}{2\|b_k\|} \sin\left( \frac{\|b_k\|}{2} \right)$,\\ $\frac{b_{k,j}b_k}{2\|b_k\|^2} \cos\left( \frac{\|b_k\|}{2} \right) + \left( \frac{e_j}{\|b_k\|} - \frac{b_{k,j} b_k}{\|b_k\|^3} \right)  \sin\left( \frac{\|b_k\|}{2} \right) \Big)$}
\EndFor
\EndIf
\EndFor
\For{$k=0,1,\dots,N-1$} \Comment{\parbox[t]{.4\linewidth}{Loop over elements $(t_k,t_{k+1})$.}}
\State $h = t_{k+1} - t_k$
\For{$p=1,2,\dots,P$} \Comment{\parbox[t]{.4\linewidth}{Loop over quadrature points $\xi_p$.}}
\State $q = \hspace{1.8em} \sum_{i=0}^1 \Big( \phi_i(\xi_p) u_{k+i} + h \psi_i(\xi_p) u_{k+i} (0,\omega_{k+i}) \Big)$
\State $\dot{q} = h^{-1} \sum_{i=0}^1 \left( \dot{\phi}_i(\xi_p) u_{k+i} + h\dot{\psi}_i(\xi_p) u_{k+i} (0,\omega_{k+i}) \right)$
\State $\ddot{q} = h^{-2} \sum_{i=0}^1  \left( \ddot{\phi}_i(\xi_p) u_{k+i} + h \ddot{\psi}_i(\xi_p) u_{k+i} (0,\omega_{k+i}) \right)$
\State $\omega = q^{-1}\dot{q}$
\State $\alpha = \mathrm{Im}\left( q^{-1}\ddot{q} - \omega^2 \right)$
\For{$j=1,2,3$}
\State $g_{3(Pk+p-1)+j} = \sqrt{hW_p} \alpha_{j+1}$
\EndFor
\For{$i=0,1$} \Comment{\parbox[t]{.48\linewidth}{Loop over variables on which $\alpha$}}
\For{$\ell = 1,2,\dots,n_{k+i}+3$} \Comment{\parbox[t]{.48\linewidth}{depends, namely, $\{y_k,\omega_k,y_{k+1},\omega_{k+1}\}$,}}
\If{$\ell \le n_{k+i}$} \Comment{\parbox[t]{.48\linewidth}{where $y_k \in \{b_k,\beta_k\}$ and}}
\If{$n_{k+i}=1$} \Comment{\parbox[t]{.48\linewidth}{$y_{k+1} \in \{b_{k+1},\beta_{k+1}\}$.}}
\State $\delta u_{k+i} = \frac{\partial u_{k+i}}{\partial \beta_{k+i}}$
\Else
\State $\delta u_{k+i} = \frac{\partial u_{k+i}}{\partial b_{k+i,\ell}}$
\EndIf
\State $\delta q = \hspace{2.3em} \phi_i(\xi_p) \delta u_{k+i} + h \psi_i(\xi_p) \delta u_{k+i} (0,\omega_{k+i})$
\State $\delta \dot{q} = h^{-1} \left( \dot{\phi}_i(\xi_p) \delta u_{k+i} + h\dot{\psi}_i(\xi_p) \delta u_{k+i} (0,\omega_{k+i}) \right)$
\State $\delta \ddot{q} =  h^{-2} \left( \ddot{\phi}_i(\xi_p) \delta u_{k+i} + h\ddot{\psi}_i(\xi_p) \delta u_{k+i} (0,\omega_{k+i}) \right)$
\Else
\State $\delta\omega_{k+i} = e_{\ell-n_{k+i}}$
\State $\delta q = \hspace{1.1em} h \psi_i(\xi_p) u_{k+i} (0,\delta\omega_{k+i})$
\State $\delta \dot{q} = \hspace{1.6em}  \dot{\psi}_i(\xi_p) u_{k+i} (0,\delta\omega_{k+i})$
\State $\delta \ddot{q} = h^{-1} \ddot{\psi}_i(\xi_p) u_{k+i} (0,\delta\omega_{k+i})$
\EndIf
\State $\delta\omega = q^{-1}\delta \dot{q} - q^{-1}\delta q \, \omega$
\State $\delta \alpha = \mathrm{Im}\left( q^{-1} \delta \ddot{q} -q^{-1}\delta q \, q^{-1} \ddot{q}  - \omega \delta \omega - \delta \omega \omega \right)$ \Comment{\parbox[t]{.2\linewidth}{Derivative of $\alpha$}} 
\State \Comment{\parbox[t]{.2\linewidth}{w.r.t. current}}
\State \Comment{\parbox[t]{.2\linewidth}{variable.}}
\State $m = \ell + \sum_{r=0}^{k+i-1} (n_r+3)$ \Comment{\parbox[t]{.32\linewidth}{Index of current variable.}}
\For{$j=1,2,3$}
\State $J_{3(Pk+p-1)+j,m} = \sqrt{hW_p} \, \delta \alpha_{j+1}$
\EndFor
\EndFor
\EndFor
\EndFor
\EndFor
\end{algorithmic}
\end{myenv}
\hrule

\section{Conclusion}

This paper has studied a family of schemes for interpolating $SO(n)$-valued functions with the aid of an embedding.  We used these schemes to construct a numerical method for computing minimum acceleration curves on $SO(n)$.  Numerical experiments indicate that the numerical solutions produced in this fashion converge optimally to the exact solution under refinement.  We did not establish this theoretically, but the interpolation error estimates in Section~\ref{sec:theory} are a first step in that direction.  We also did not address the well-posedness of the minimum acceleration problem~(\ref{minaccel}).  Both of these topics are worthy of further study.  In fact, it seems worthwhile to pursue a systematic study of weak formulations of the equations that govern Riemannian cubics, as well as their discretization with manifold-valued finite elements.  A major step in this direction has been performed in~\cite{grohs2015optimal}, where the authors have presented the aforementioned theory not for Riemmannian cubics (a second-order variational problem), but for first-order variational problems involving manifold-valued functions.

\bibliographystyle{plain}
\bibliography{references}

\end{document}